\documentclass[reqno,11pt]{amsart}

\usepackage{graphicx} 
\usepackage{fullpage}
\usepackage{url}
\usepackage{mathrsfs}
\usepackage{amsmath} 
\usepackage{amsthm}
\usepackage{amsfonts}
\usepackage{amssymb}
\usepackage{ulem}
\usepackage{esint}
\usepackage{pstricks,pst-node,pst-text,pst-3d}
\usepackage{graphics,graphicx,pspicture,graphpap}

\usepackage{enumerate}

\usepackage{hyperref}

\newtheorem{theorem}{Theorem}[section]
\newtheorem{lemma}[theorem]{Lemma}
\newtheorem{proposition}[theorem]{Proposition}
\newtheorem{corollary}[theorem]{Corollary}
\newtheorem{cor}[theorem]{Corollary}

\newtheorem{problem}[theorem]{{\large Problem}}

\newtheorem*{proposition*}{Proposition}

\theoremstyle{remark}
\newtheorem{remark}[theorem]{Remark}
\newtheorem{definition}[theorem]{Definition}

 \numberwithin{equation}{section}

\newcommand{\zero}{{\bf 0}}

\newcommand {\bH}{\mathbb {H}}

\newcommand {\bR}{\mathbb {R}}
\newcommand {\R}{\mathbb{R}}
\newcommand {\N}{\mathbb{N}}
\newcommand {\bX}{\mathbb {X}}

\newcommand {\cH}{\mathcal {H}}

\newcommand {\cB}{\mathcal {B}}

\newcommand {\cD}{\mathcal {D}}
\newcommand {\cC}{\mathcal {C}}
\newcommand {\cL}{\mathcal {L}}
\newcommand {\cI}{\mathcal {I}}
\newcommand {\bP}{\mathbb{P}}

\newcommand {\Z}{\mathbb {Z}}
\newcommand{\ahlfors}{C_{\rm dens}}

\DeclareMathOperator{\diam}{diam}

\DeclareMathOperator{\dom}{dom}

\newcommand{\sepconst}{\eta}
\newcommand{\pad}{\beta}
\newcommand{\Ctwo}{C_2}
\newcommand{\costconst}{{C_3}}
\newcommand{\tailconst}{{C_5}}
\newcommand{\bddgeom}{{C_4}}
\newcommand{\Csix}{C_6}
\newcommand{\cA}{\mathcal{A}}

\newcommand{\x}{{\bf x}}
\newcommand{\y}{{\bf y}}
\newcommand{\cS}{\mathcal{S}}

\newcommand{\cT}{\mathcal{T}}
\renewcommand{\epsilon}{\varepsilon}
\renewcommand{\phi}{\varphi}

\DeclareMathOperator{\Lip}{Lip}

\title{Characterizing rectifiability via biLipschitz pieces of Lipschitz mappings on the space}

\author{Sean Li and
Raanan Schul}
\date{\today}

\thanks{R. Schul is supported by the National Science Foundation under Grant No. DMS-2154613} 

\begin{document}
\maketitle

\begin{abstract}
We give the following  characterization of rectifiable metric spaces.
A metric space with positive lower Hausdorff density is rectifiable if and only if, for any subset $F$ and $f:F\to Y$, a Lipschitz map into  a metric space with  positive measure image (of the same dimension), there exists a positive measure subset $A\subset F$ so that $f$ is biLipschitz on $A$. We also give a characterization in terms of a full biLipschitz decomposition. These characterizations are new even for subsets of Euclidean space.

One of our tools is Alberti representations. On the way we give a method for constructing independent Alberti representations, which may be of independent interest. We use this to characterize unrectifiable metric spaces as those spaces for which there exist a positive measure subset $S$ and a Lipschitz map $\phi$ into a lower dimensional Euclidean space so that $S$ is $\cH^1$-null with respect to all curve fragments that are quantitatively transversal to $\phi$.
\end{abstract}

\section{Introduction}

 The following Sard-like result is implicit in \cite{Kirchheim-MD}.

\begin{lemma} \label{l:kirchheim}
    Let $f : A \to Y$ be Lipschitz where $A \subseteq \R^p$ is Borel and $Y$ is a metric space. Then there exist a Borel decomposition 
    $$A = N \cup \bigcup_{i \in \N} E_i$$
    so that $\cH^p(f(N)) = 0$ and $f|_{E_i}$ is biLipschitz for each $i$.
\end{lemma}

This is sometimes called a biLipschitz decomposition of $f$. We give the proof at the end of the section.

Viewing $N$ as the set of critical points, we get that the critical values $f(N)$ are measure zero. As $\cH^p(f(A \setminus \bigcup_i E_i)) = 0$, any positive measure of the image must then come from the biLipschitz pieces $f|_{E_i}$.
 This was also seen quantitatively  in  \cite{Schul-lip-bilip}.
 
 Recall that a metric space $X$ is said to be $p$-rectifiable if it can be covered, up to a set of $\cH^p$-measure zero, by a countable number of Lipschitz images of Borel subsets of $\R^p$. By Lemma \ref{l:kirchheim}, we see that these parameterizing maps can even be taken to be biLipschitz. On the other hand, $X$ is purely $p$-unrectifiable if all of its $p$-rectifiable subsets have $\cH^p$-measure zero. If $p \notin \N$, then $X$ is automatically purely $p$-unrectifiable. An easy conclusion is that any Lipschitz map on a rectifiable space $X$ must admit a biLipschitz decomposition.
We elaborate on this in  Corollary \ref{c:cor2} below. In fact, Corollary \ref{c:cor2} will show that the biLipschitz decomposition phenomenon {\it characterizes} rectifiability.

On the other hand, it was shown in \cite{LDLR-2017} that this property does not hold for the Heisenberg group endowed with the Carnot-Carath\'eodory metric. In particular, the Heisenberg group $\mathbb{H}$ is Ahlfors 4-regular and there exists a metric space $\bX$ and a Lipschitz map $f : \mathbb{H} \to \bX$ with positive $\cH^4$-measure image so that $f|_A$ is not biLipschitz for any positive measure $A \subseteq \mathbb{H}$.  In fact, it was further proven in \cite{LDLR-2017} that $\bH$ is not minimal in looking down (which is defined and discussed below).

 This is also related to the notion of strongly unrectifiable set introduced by Ambrosio and Kirchheim \cite{Ambrosio-Kirchheim}. See 
 also the introduction of \cite{bate2020purely}.

In this paper we aim to expand on \cite{LDLR-2017}. We provide a converse to Lemma \ref{l:kirchheim} in the following sense. If metric space valued Lipschitz maps on subsets of a space $X$ always admit biLipschitz decompositions, then $X$ is rectifiable. Our main result is the following.

\begin{theorem}\label{th:main}
     Let $(X,d)$ be a compact metric space so that $0 < \cH^p(X) < \infty$, $\theta_*^p(X,x) > 0$ for $\cH^p$-a.e. $x \in X$, and $X$ is not $p$-rectifiable.
 Then there is a positive measure $F \subset X$, a metric space $(Y,d)$, and a Lipschitz map $f:F \to Y$ such that  $\cH_d^p(f(F)) > 0$ and $f|_{A}$ is not biLipschitz for all positive measure $A \subseteq F$.
\end{theorem}

The way the above theorem is proven is by judiciously finding a set $F\subset E$ (the set $F$ ends up being compact,  purely unrectifiable and with a uniform lower density), constructing a new metric on $F$, and then taking $f$ as the identity map.

Theorem \ref{th:main} allows us to characterize rectifiability by the biLipschitz decomposition phenomenon.

 \begin{cor}\label{c:cor2}
 Let $(E,\rho)$ be a metric space with $0 < \cH^p(E) < \infty$ and $\theta_*^p(E,x) > 0$ for $\cH^p$-a.e. $x \in E$.
Then $E$ is p-rectifiable if and only if for any $F\subset E$ and Lipschitz $f:F\to (Y,d)$, there is a Borel decomposition $F = N \cup E_1 \cup E_2 \cup ...$ so that $\cH^p(f(N)) = 0$ and $f|_{E_i}$ is biLipschitz for each $i$.  
 \end{cor}

Loosely speaking, Corollary \ref{c:cor2} says that one may characterize rectifiability by asking the question: is there a Lipschitz map that pushes forward positive Hausdorff measure and is not essentially biLipschitz? For rectifiable spaces the answer is NO and for unrectifiable spaces the answer is YES.

\begin{proof}[Proof of Corollary \ref{c:cor2}]
    If $E$ is unrectifiable, then Theorem \ref{th:main} gives a positive measure subset $F$ and a Lipschitz $f : F \to Y$ with positive measure image and no biLipschitz pieces. If there were a biLipschitz decomposition $F = N \cup E_1 \cup E_2 \cup ...$ then for all $i$, $\cH^p(E_i) = 0$ so that $\cH^p(f(E_i)) = 0$. This would mean $\cH^p(f(N)) = \cH^p(f(F)) > 0$, a contradiction.
    
The other direction is essentially Lemma \ref{l:kirchheim}, but we provide the details. Let
$E$ be rectifiable 
and $f : F \to Y$ be Lipschitz where $F \subset E$ has positive $\cH^p$-measure. 
As $F$ is then also rectifiable, there exists a countable number of biLipschitz maps $g_i : A_i \to F$ where $A_i \subset \R^p$ is Borel so that $\cH^p(F \setminus \bigcup_i g_i(A_i)) = 0$. It then suffices to find a biLipschitz decomposition of each $g_i(A_i)$.

By Lemma \ref{l:kirchheim}, there exists a Borel decomposition $A_i = N \cup F_1 \cup F_2 \cup ...$ for each function $f \circ g_i : A_i \to Y$. We then have
$$g(A_i) = g_i(N) \cup \bigcup_{j} g_i(F_j).$$
We claim this is the biLipschitz decomposition.

First note that all $g_i(F_j)$ are Borel as $g_i$ is biLipschitz. By assumption, $\cH^p(f(g_i(N))) = 0$. On the other hand, $g_i$ and $f \circ g_i|_{F_j}$ are biLipschitz for all $j$. Thus, $f|_{g_i(F_j)} = f \circ g_i|_{F_j} \circ g_i^{-1}|_{g_i(F_j)}$ is also biLipschitz.
\end{proof}

Corollary \ref{c:cor2} shows that rectifiability is related to finding a full biLipschitz decomposition of $f$. It is also easy to see that it is equivalent to finding a single biLipschitz piece when $f$ has positive measure image.

 \begin{cor}\label{c:cor1}
 Let $(E,\rho)$ be a metric space with $0 < \cH^p(E) < \infty$ and $\theta_*^p(E,x) > 0$ for $\cH^p$-a.e. $x \in E$.
Then $E$ is p-rectifiable if and only if for any $F\subset E$ and Lipschitz $f:F\to (Y,d)$ with positive $\cH^p$-measure $f(F)$
there is $A\subset F$ of positive $\cH^p$-measure so that $f|_A$ is biLipschitz.  
 \end{cor}

 \begin{proof}
    If $E$ is unrectifiable, then the corollary follows immediately from Theorem \ref{th:main}.

    If $E$ is rectifiable and $f : F \to Y$ is as described, then by Corollary \ref{c:cor2}, there is a Borel decomposition $F = N \cup E_1 \cup E_2 \cup...$ so that $\cH^p(f(N)) = 0$ and each $f|_{E_i}$ is biLipschitz. As $\cH^p(f(F)) > 0$, it cannot be that $\cH^p(F \setminus N) = 0$. Thus, there is some $E_i$ of positive measure, as desired.
 \end{proof}

We remark on the relationship between this paper and the work in \cite{LDLR-2017} where  the domain is the Heisenberg group.  Both papers construct new metrics on the space by collapsing points, but there are two main differences.
The first difference is that as our spaces in question are not homogeneous, we are forced to vary the scales of collapse based on location.
The second difference is that the metric of the Heisenberg group snowflakes in the vertical direction. This leads to a lot of slack in  the triangle inequality, which is what \cite{LDLR-2017} used to collapse points. Our settings do not have this characteristic and so we instead collapse points in directions that exhibit no rectifiability. The existence of these directions is partially a result of the following proposition.

\begin{proposition} \label{p:intro-extend}
  Let $(X,d,\mu)$ be a compact metric measure space and $\phi : X \to \R^n$ a Lipschitz map that has $n$ $\phi$-independent Alberti representations. Then there is a Borel decomposition
  \begin{align*}
    X = \bigcup_{i \in \N} A_i \cup \bigcup_{i \in \N} S_i,
  \end{align*}
  where for each $i \in \N$ there is a $\kappa_i > 0$, and Lipschitz $\phi_i : X \to \R^{n+1}$ for which $\phi_i|_{A_i}$ has $n+1$ $\phi_i$-independent Alberti representations and $\cH^1(S_i \cap \mathrm{Im}(\gamma)) = 0$ for all biLipschitz $\gamma : K \to X$ where $K \subset \R$ is compact and
  $$\|(\phi \circ \gamma)'(t)\| \leq \kappa_i \Lip(\phi,\gamma(t)) \Lip(\gamma,t)$$
  for almost every $t \in K$.
\end{proposition}

All the terminology will be given in Section \ref{s:preliminaries}. From this proposition, we get the following characterization of unrectifiability in terms of curve fragments ``transversal'' to some Lipschitz map into a lower dimensional Euclidean space.

\begin{theorem} \label{th:intro-null-char}
Let $(X,d)$ be a compact metric space so that $0 < \cH^p(X) < \infty$. Then $X$ is $p$-unrectifiable if and only if there is a $\kappa > 0$, a $S \subset X$ for which $\cH^p(S) > 0$, and a Lipschitz map $\phi : S \to \R^n$ where $n < p$ so that $\cH^1(\mathrm{Im}(\gamma) \cap S) = 0$ for every biLipschitz $\gamma : K \to X$ where $K \subset \R$ is compact and
  $$\|(\phi \circ \gamma)'(t)\| \leq \kappa \Lip(\phi,\gamma(t)) \Lip(\gamma,t)$$
  for almost every $t \in K$.
\end{theorem}

The transverse of $\phi$ will then be the directions that exhibit no rectifiability. 

In the case when we are dealing with subsets of some Euclidean space $\R^P$, we could instead use the following result, which is easier to understand.

\begin{theorem} \label{p:david}
Let $p > 0$, $P \geq p$ be an integer, $n$ be the largest integer smaller than $p$, and $E\subset \R^P$ a purely $p$-unrectifiable subset for which $0 < \cH^p(E) < \infty$. Then for any $\kappa > 0$, we have a countable Borel decomposition
$$E = \bigcup_i E_i$$
  and a collection of $n$-dimensional planes $W_i \subset \R^P$ so that $\cH^1(E_i \cap \mathrm{Im}(\gamma)) = 0$ for all biLipschitz $\gamma : K \to \R^P$ where $K \subset \R$ is compact and
  $$d(\gamma'(t),W_i) \geq (1-\kappa) \|\gamma'(t)\|$$
  for almost every $t \in K$.
\end{theorem}

This is simply Theorem 2.21 of \cite{bate2020purely} applied to the inclusion map $E \hookrightarrow \R^P$. As rectifiable spaces all have positive lower density, Theorem \ref{p:david} immediately allows us to get the following Euclidean variant of Theorem \ref{th:intro-null-char}, which says that we can restrict the maps $\phi$ to only the orthogonal projections onto lower dimensional subspaces.

\begin{corollary} \label{c:euc-char}
Let $p > 0$, $P \geq p$ be an integer, $n$ be the largest integer smaller than $p$, and $E \subset \R^P$ be so that $0 < \cH^p(E) < \infty$. Then $E$ is $p$-unrectifiable if and only if there is a $\kappa > 0$, a subset $S \subset E$ for which $\cH^p(S) > 0$, and a $n$ dimensional subspace $W$ so that $\cH^1(S \cap \mathrm{Im}(\gamma)) = 0$ for all biLipschitz $\gamma : K \to \R^P$ where $K \subset \R$ is compact and
  $$d(\gamma'(t),W) \geq (1-\kappa) \|\gamma'(t)\|$$
  for almost every $t \in K$.
\end{corollary}

We remark that the backwards direction of Corollary \ref{c:euc-char} is not quite Remark 2.22 of \cite{bate2020purely}. There it says that a $p$-rectifiable piece $S$ is not $\cH^1$-null to transversal fragments to $n$ dimensional subspaces {\em after} being pulled back to a parameterizing $\R^p$. In Corollary \ref{c:euc-char}, we are asking that $S$ is not $\cH^1$-null to transversal fragments of $n$ dimensional subspaces in the ambient $\R^P$.

Finding biLipschitz pieces of Lipschitz maps is related to  the concept of being {\it minimal in looking down} as discussed in 
 \cite{heinonen1997thirty, david1997fractured}.
An Ahlfors $p$-regular metric space $(X,d)$ is said to be a BPI (``big pieces of itself'') space if there are constants $C \geq 1$, $\theta > 0$ so that for all $x_1,x_2 \in X$ and $0 < r_1,r_2 < \diam(X)$, there is a closed set $A \subset B(x_1,r_1)$ so that $\cH^p(A) \geq \theta r_1^p$ and there is a $C$-biLipschitz map $f : (A,\tfrac{r_2}{r_1}d) \to B(x_2,r_2)$.

Given two Ahlfors $p$-regular BPI spaces $X$ and $Y$, $X$ is said to {\em look down on} $Y$ if there is a positive measure closed subset $A \subset X$ and a Lipschitz function $f : A \to Y$ with positive $\cH^p$ image. If $Y$ also looks down on $X$, then $X$ and $Y$ are said to be {\em look-down equivalent}. A BPI space is said to be {\em  minimal in looking down} if it is look-down equivalent to any other space that it looks down upon.

The fact that Lipschitz maps on rectifiable metric spaces admit biLipschitz pieces immediately gives that rectifiable BPI spaces are minimal in looking down. However, Theorem \ref{th:main} does not necessarily imply that a purely unrectifiable BPI space $Y$ is not minimal in looking down since it is possible that there are other Lipschitz maps $Y \to X$ of positive $\cH^p$ image. As mentioned before, it was shown in \cite{LDLR-2017} that $\bH$, a BPI space, is not minimal in looking down. We thus have the following problem.
\begin{problem}
    Are rectifiable BPI spaces the {\em only} BPI spaces that are minimal in looking down?
\end{problem}

The only place we use the lower density assumption is in the application of \eqref{eq:Fr-defn} in the proof of Lemma \ref{l:not-bilip}. We thus also ask if the lower density assumption is necessary.

\begin{problem}
Is Theorem \ref{th:main} still true if one removes the $\cH^p$ almost everywhere $\theta_*^p(X,x) > 0$ hypothesis?
\end{problem}

We end by proving Lemma \ref{l:kirchheim}.

\begin{proof}[Proof of Lemma \ref{l:kirchheim}]
    We will not cover the notion of metric differential from \cite{Kirchheim-MD} as it is used only in this proof.

    First assume $A =  \R^p$. Lemma 4 of \cite{Kirchheim-MD} gives a decomposition $\R^p = N_1 \cup N_2 \cup E_1 \cup E_2 \cup ...$ where $f|_{E_i}$ are all biLipschitz, the metric differential does not exist for points of $N_1$, and the metric differential exists but is not a norm for points of $N_2$. By Theorem 1 of \cite{Kirchheim-MD}, $\cH^p(N_1) = 0$ so $\cH^p(f(N_1)) = 0$. Theorem 7 of \cite{Kirchheim-MD} (and especially its proof where it is shown that $\mathscr{J}(MD(f,x)) = 0$ when $x \in \mathscr{MD}(f) \setminus \mathscr{MD}_r(f)$) gives that $\cH^p(f(N_2)) = 0$. This finishes the $A = \R^p$ case.

    For general $A \subset \R^p$, we can first view $Y$ as a subset of $\ell_\infty$ via the Kuratowski embedding (after first truncating to the image of $f$ if necessary). We can then use the McShane-Whitney extension to Lipschitz extend each coordinate function of $f$ to $F : \R^p \to \ell_\infty$. Intersecting the biLipschitz decomposition of $F$ with $A$ then gives the lemma.
\end{proof}

\section{Preliminaries} \label{s:preliminaries}

All balls are open.

Let $(X,\rho)$ be a metric space. Following \cite{LDLR-2017}, we call a symmetric function $c: X \times X \to [0,\infty)$ such that $c \leq \rho$ a {\em cost function}. We define the set of {\em shortcuts} as
\begin{align*}
  \cS = \{(x,y) \in X \times X : c(x,y) < \rho(x,y)\}.
\end{align*}

Let $\cI$ denote the set of all finite length sequences in $X$. We call these itineraries. Given $x,y \in X$, we let $\cI(x,y) = \{(x_1,x_2,...,x_N) \in \cI : x_1 = x, x_N = y\}$. Given a sequence $\x = (x_1,...,x_N) \in \cI$, we let
\begin{align*}
  S(\x) &= \{(x_i,x_{i+1}) : 1\leq i \leq N, (x_i,x_{i+1}) \in \cS\}, \\
  S^c(\x) &= \{(x_i,x_{i+1}) : 1\leq i \leq N, (x_i,x_{i+1}) \notin \cS\},
\end{align*}
denote the shortcut and non-shortcut segments and
\begin{align*}
  c(\x) = \sum_{i=1}^{N-1} c(x_i,x_{i+1})
\end{align*}
denote the cost of the sequence.

We can now define
\begin{align*}
  d : X \times X &\to [0,\infty), \\
  (x,y) &\mapsto \inf \{c(\x) : \x \in \cI(x,y)\}.
\end{align*}
That $d$ is symmetric and satisfies the triangle inequality follows easily, but it may be that $d(x,y) = 0$ for some $x \neq y$. Nevertheless, we will still define balls with respect to $d$ as
\begin{align*}
  B_d(x,r) := \{y \in X : d(x,y) < r\}.
\end{align*}
It also easily follows that $d \leq \rho$ so that $B_\rho(x,r) \subseteq B_d(x,r)$.

An itinerary $(x_1,...,x_N)$ is said to be {\it alternating} if $N \in 2\N$, $(x_i,x_{i+1}) \in \cS$ if and only if $i \in 2\N$, and any shortcut is used at most once. We let $\cI_A(x,y)$ denote all the alternating itineraries from $x$ to $y$. A simple triangle inequality argument then gives that
\begin{align*}
  d(x,y) = \inf \{c(\x) : \x \in \cI_A(x,y)\},
\end{align*}
and so it suffices to study only itineraries from $\cI_A$.

Given a function $f : (X,d_X) \to (Y,d_Y)$ and $x \in X$, we define the pointwise Lipschitz constant as
\begin{align*}
  \Lip(\gamma,x) := \limsup_{y \to x} \frac{d_Y(f(x),f(y))}{d_X(x,y)}.
\end{align*}
This satisfies a chain rule inequality $\Lip(f \circ g,x) \leq \Lip(f,g(x)) \Lip(g,x)$.

Given a Lipschitz $\phi : X \to \R^n$ and a $0 < \theta < 1$, we define the $\theta$-transversal cone of $\phi$ at $x$ as
\begin{align*}
  E_\phi(x,\theta) := \{y \in X : |\phi(x) - \phi(y)| \leq \theta \rho(x,y)\}.
\end{align*}

\begin{lemma} \label{l:transverse-point}
  Let $0 < \theta < 1$, $0 < k < p$, and $\phi : X \to \R^k$ be Lipschitz. Then, for $\cH^p$-a.e. $x \in X$, we have that
  \begin{align*}
    E_\phi(x,\theta) \cap (B(x,r) \setminus \{x\}) \neq \emptyset, \qquad \forall r > 0.
  \end{align*}
\end{lemma}
Note that this does not use that $E$ is purely unrectifiable.
\begin{proof}
  Suppose for contradiction that there is some $E \subset X$ so that $\cH^p(E) > 0$ and for all $x \in E$, there is some $r_x > 0$ so that $E_\phi(x,\theta) \cap (B(x,r_x) \setminus \{x\}) = \emptyset$ for all $x \in E$. After possibly passing to a positive measure subset of $E$, we may suppose that there is some $r > 0$ so that $r_x \geq r$ and $\diam(E) < r$. Then for any pair of distinct points $x,y \in E$, we cannot have that $x \in E_\phi(y,\theta)$ so that 
  \begin{align*}
    |\phi(x) - \phi(y)| > \theta \rho(x,y), \qquad \forall x,y \in E.
  \end{align*}
  As $\phi$ is Lipschitz, we then get that $\phi$ is biLipschitz on $E$. But we cannot have a biLipschitz map from a set of positive $\cH^p$-measure into $\R^k$ when $k < p$, so we get a contradiction.
\end{proof}

Applying Lemma \ref{l:transverse-point} to a sequence $r_n \to 0$ immediately gives the following corollary.

\begin{corollary} \label{c:transverse-point}
  Let $0 < \theta < 1$, $0 < k < p$, and $\phi : X \to \R^k$ be Lipschitz. For $\cH^p$-a.e. $x \in X$, there is a sequence $\{x_n\}_{n \in \N} \subset X \setminus \{x\}$ so that $x_n \to x$ so that $x_n \in E_\phi(x,\theta)$ for all $n$.
\end{corollary}

\subsection{Alberti representations}

Given a metric space $X$, a (curve) fragment is a biLipschitz function $f : K \to X$ where $K \subset \R$ is compact. We let $\Gamma(X)$ denote all the fragments of $X$. By identifying a $\gamma \in \Gamma(X)$ via its compact graph in $\R \times X$, we can metrize $\Gamma(X)$ using the Hausdorff metric. We also let $\Pi(X)$ denote all the fragments of $X$ whose domain is a closed interval. Note that $\Pi(X)$ may be trivial for topological reasons.
We will often identify a fragment with its image; for example we will write $\cH^1(\gamma)$ etc. 

Given any Lipschitz $\gamma : D \to X$ where $D \subset \R$ is Borel, $\Lip(\gamma,x)$ is defined and bounded for almost every $x \in D$. We then define its parametric length as
\begin{align*}
  \ell(\gamma) := \int_D \Lip(\gamma,x) ~dx.
\end{align*}
The area formula gives that $\ell(\gamma)\geq \cH^1(\gamma)$ but there is no quantitative reverse bound. However, we do have that $\cH^1(\gamma)=0$ then $\ell(\gamma)=0$.

Given a Lipschitz map $\phi : X \to \R^n$ and $0 < \kappa < 1$, we define
$$T_\kappa(\phi) = \{\gamma \in \Gamma(X): \|(\phi \circ \gamma)'(t)\| \leq \kappa \Lip(\phi,\gamma(x)) \Lip(\gamma,x)~ a.e. ~t \in \dom(\gamma)\}.$$
Note that if $n = 0$, then $T_\kappa(\phi) = \Gamma(X)$.

\begin{definition}
  Let $\bP$ be a probability measure on $\Gamma(X)$ and for each $\gamma \in \Gamma(X)$, let $\mu_\gamma$ be a measure on $X$ so that $\mu_\gamma \ll \cH^1|_\gamma$ and $\gamma \mapsto \mu(Y)$ is Borel measurable for every Borel $Y \subset X$. The pair $(\bP,\{\mu_\gamma\}_{\gamma})$ is an {\em Alberti representation} of a measure $\mu$ on $X$ if for every Borel $B \subset X$, we have
  \begin{align*}
    \mu(B) = \int_{\Gamma(X)} \mu_\gamma(B) ~d\bP(\gamma).
  \end{align*}
  Given an Alberti representation $\cA = (\bP,\{\mu_\gamma\})$, we say a property holds for almost every $\gamma \in \cA$ if it holds for $\bP$-a.e. $\gamma \in \Gamma(X)$.
\end{definition}

Following Definition 5.7 of \cite{bate-jams}, given a Lipschitz $\phi : X \to \R^n$ and a $\delta > 0$, we say that $\gamma \in \Gamma(X)$ has $\phi$-speed $\delta$ if for almost every $t \in \dom(\gamma)$,
\begin{align*}
  \|(\phi \circ \gamma)'(t)\| \geq \delta \Lip(\phi,\gamma(x)) \Lip(\gamma,x).
\end{align*}
Note that this is essentially the opposite condition of $T_\delta(\phi)$.
We say an Alberti representation $\cA$ has $\phi$-speed $\delta$ if almost every $\gamma \in \cA$ has $\phi$-speed $\delta$.

Given $w \in \R^n$ and $0 < \theta < 1$, we define the cone of width $\theta$ centered on $w$ to be the set
\begin{align*}
  C(w,\theta) = \{v \in \R^n : w \cdot v \geq (1-\theta)\|v\|\}.
\end{align*}
Given a Lipschitz $\phi : X \to \R^n$ and a cone $C \subset \R^n$, we say $\gamma \in \Gamma(X)$ is in the $\phi$-direction $C$ if for almost every $t \in \dom(\gamma)$, we have
\begin{align*}
  (\phi \circ \gamma)'(t) \in C.
\end{align*}
An Alberti representation $\cA$ is in $\phi$-direction $C$ if almost every $\gamma \in \cA$ is in $\phi$-direction $C$.

Cones $C_1, ..., C_m \subset \R^n$ are linearly independent if for any nonzero vectors $v_i \in C_i$, we have that $v_1,...,v_m$ are linearly independent. Given a Lipschitz $\phi : X \to \R^n$, we say that Alberti representations $\cA_1,...,\cA_m$ are $\phi$-independent if there are linearly independent $C_1,...,C_m \subset \R^n$ so that each $\cA_i$ is in $\phi$-direction $C_i$. We say a measure $\mu$ has $n$ independent Alberti representations if there is a Borel decomposition $X = \bigcup_{i \in \N} A_i$ and Lipschitz maps $\phi_i : A_i \to \R^n$ where each $\mu|_{A_i}$ has $n$ $\phi_i$-independent Alberti representations.

\section{Constructing independent Alberti representations}

We begin by defining a family of singular sets with respect to Lipschitz maps.

\begin{definition}
    Let $\phi : X \to \R^n$ be Lipschitz. For any $\kappa > 0$, we define $\tilde{D}(\phi,\kappa)$ to be the set of all $S \subset X$ for which $\cH^1(S \cap \gamma) = 0$ for all $\gamma \in T_\kappa(\phi)$. We define $\tilde{D}(\phi)$ to be the sets $S$ for which there exists a Borel decomposition $S = S_1 \cup S_2 \cup ...$ and positive $\kappa_1,\kappa_2,...$ so that $S_i \in \tilde{D}(\phi,\kappa_i)$.
    
    Finally, given $n \in \N$, we define $\tilde{D}(n)$ to be the sets $S \subset X$ for which there is a Borel decomposition $S = S_1 \cup S_2 \cup ...$ and Lipschitz $\phi_i : S_i \to \R^{n_i}$ where $n_i \leq n$ and $S_i \in \tilde{D}(\phi_i)$.
\end{definition}

The following proposition lets us construct additional Alberti representations that are independent to an existing collection. This can be viewed as an intrinsic version of Proposition 2.10 of \cite{bate2020purely}. One can see that the proposition is just a restatement of Proposition \ref{p:intro-extend}.

\begin{proposition} \label{p:extend-alberti}
  Let $(X,d,\mu)$ be a compact metric measure space and $\phi : X \to \R^n$ a Lipschitz map that has $n$ $\phi$-independent Alberti representations. Then there is a Borel decomposition
  \begin{align*}
    X = S \cup \bigcup_{i \in \N} A_i,
  \end{align*}
  where $S \in \tilde{D}(\phi)$ and for each $i \in \N$ there is a Lipschitz $\phi_i : X \to \R^{n+1}$ for which $\phi_i|_{A_i}$ has $n+1$ $\phi_i$-independent Alberti representations.
\end{proposition}

The following lemma will allow us to find additional coordinate functions.

\begin{lemma} \label{l:banach-partition}
  Let $Y$ be a Banach space for which there is a countable subset $\{\xi_i\}_i \subset \overline{B_{Y^*}}$ so that 
  \begin{align}
    \|x\| = \sup_{i \in \N} \xi_i(x), \qquad \forall x \in Y. \label{eq:norming}
  \end{align}
  Then for every $\epsilon > 0$ and Lipschitz $\gamma : D \to Y$ where $D$ is Borel, there is a Borel decomposition $D = N \cup E_1 \cup E_2 \cup ...$ where $\mathcal{L}(N) = 0$ and $(\xi_i \circ \gamma)'(t) \geq (1-\epsilon) \Lip(\gamma,t)$ for a.e. $t \in E_i$.
\end{lemma}

\begin{proof}
  Suppose without loss of generality that $\gamma$ is 1-Lipschitz. We will show for every positive measure subset of $S \subset D$, there is a Borel $E \subset S$ of positive measure for which there is some $i$ so that $(\xi_i \circ \gamma)'(t) \geq (1-\epsilon) \Lip(\gamma,t)$ for all $t \in E$. The lemma follows from this by standard measure theoretic techniques.

  First, as $t \mapsto \Lip(\gamma,t)$ is measurable, we can find a point $x \in S$ that is both a density point of $S$ and a Lebesgue point of $\Lip(\gamma,\cdot)$. Let $m = \Lip(\gamma,x)$ and choose $r_0 > 0$ sufficiently small so that
  \begin{align}
    |[x-r,x+r] \setminus S| &< \tfrac{\epsilon}{10} r, \label{eq:S-density} \\ 
    \left|\left\{t \in S \cap [x-r,x+r] : \Lip(\gamma,t) > m + \tfrac{\epsilon}{10} \right\}\right| &< \tfrac{\epsilon}{10} r, \label{eq:Lip-lebesgue}
  \end{align}
  for all $r < r_0$.

  Choose some $y \in S \cap [x-r_0,x+r_0]$ so that $\|\gamma(y) - \gamma(x)\| > (m - \tfrac{\epsilon}{2}) |y-x|$. Let $r = |y-x|$ and suppose without loss of generality that $y > x$. By \eqref{eq:norming}, there is some $\xi_i$ so that
  \begin{align}
    \xi_i(\gamma(y)) - \xi_i(\gamma(x)) = \xi_i(\gamma(y) - \gamma(x)) > (m - \tfrac{\epsilon}{2})r. \label{eq:big-psi}
  \end{align}

  Suppose for contradiction
  \begin{align}
    (\xi_i \circ \gamma)'(t) < (1-\epsilon) \Lip(\gamma,t) \quad \mathrm{a.e.} ~t \in S \cap [x,y]. \label{eq:big-contraction}
  \end{align}
  Note that $\xi_i \circ \gamma$ is 1-Lipschitz and so
  \begin{align*}
    \xi_i(\gamma(y)) - \xi_i(\gamma(x)) &\leq \int_{S \cap [x,y]} (\xi_i \circ \gamma)'(t) ~dt + |[x,y] \setminus S| \\
    &\overset{\eqref{eq:S-density} \wedge \eqref{eq:big-contraction}}{\leq} (1-\epsilon) \int_{S \cap [x,y]} \Lip(\gamma,t) ~dt + \tfrac{\epsilon}{10} r \\
    &\overset{\eqref{eq:Lip-lebesgue}}{\leq} (1-\epsilon) \left[ (m + \tfrac{\epsilon}{10}) r + \tfrac{\epsilon}{10} r \right] + \tfrac{\epsilon}{10} r \\
    &< (m - \tfrac{\epsilon}{2})r.
  \end{align*}
  This is a contradiction of \eqref{eq:big-psi}. Thus, \eqref{eq:big-contraction} is false, which is what we wanted to show.
\end{proof}

We now begin proving the proposition. First, we may assume by rescaling that $\phi$ is 1-Lipschitz. We may also assume $X$ is a subset of some Banach space $Y$ satisfying \eqref{eq:norming}.
This is always possible if $X$ is separable. Indeed, there is always an isometric embedding $\iota: X \hookrightarrow \ell_\infty$. One can then simply take $\{\xi_i\}_i$ to be the coordinate functionals $\xi_i(x) = x_i$ (and their negatives).

We extend $\phi$ onto $Y$ and can also assume $\phi$ is linear on $Y$. Indeed, we can further embed $X$ isometrically in $\R^n \oplus_\infty \ell_\infty$ via $x \mapsto (\phi(x),\iota(x))$. As $\phi$ is 1-Lipschitz, this embedding is isometric. We can then view $\phi$ as the linear projection map onto the $\R^n$ component. It is clear that \eqref{eq:norming} is still satisfied on $Y = \R^n \oplus_\infty \ell_\infty$.

We now fix such a $\{\xi_i\}_{i \in \N}$ satisfying \eqref{eq:norming} for the remainder of the section. Let $\tilde{X}$ be the compact convex hull of $X$.
We define for $\delta > 0$ and $i \in \N$
\begin{align*}
  \tilde{T}_{\delta} &= \left\{\gamma \in \Gamma(X) : |(\phi \circ \gamma)'(t)| < \delta \Lip(\gamma,t), ~a.e. ~t \in \dom(\gamma)\right\}, \\
  \tilde{T}_{\delta,i} &= \left\{\gamma \in \Pi(\tilde{X}) : \frac{|\phi(\gamma(s)) - \phi(\gamma(t))|}{\|\gamma(s) - \gamma(t)\|} \leq \delta, \frac{\xi_i(\gamma(t)) - \xi_i(\gamma(s))}{\|\gamma(t) - \gamma(s)\|} \geq \frac{1}{2}, ~\forall s < t \in \dom(\gamma)\right\}.
\end{align*}

It is possible for $\tilde{T}_{\delta,i}$ to be empty for certain $\delta$ and $i$.

\begin{lemma} \label{l:T-compact}
  $\tilde{T}_{\delta,i}$ is closed.
\end{lemma}

\begin{proof}
  Let $\gamma_n \in \tilde{T}_{\delta,i}$ be a sequence of Lipschitz curves converging to some $\gamma \in \Pi(\tilde{X})$. For any two $t,s \in \dom(\gamma)$, there exist $t_m,s_m \in \dom(\gamma_m)$ so that $\gamma_m(t_m) \to \gamma(t)$ and $\gamma_m(s_m) \to \gamma(s)$. We then have that
  \begin{align*}
    \delta \geq \frac{|\phi(\gamma_m(t_m)) - \phi(\gamma_m(s_m))|}{\|\gamma_m(t_m) - \gamma_m(s_m)\|} \to \frac{|\phi(\gamma(t)) - \phi(\gamma(s))|}{\|\gamma(t) - \gamma(s)\|}.
  \end{align*}
  Similarly, $\frac{\xi_i(\gamma(t)) - \xi_i(\gamma(s))}{\|\gamma(t) - \gamma(s)\|} \geq \frac{1}{2}$ for $s < t$, proving the lemma.
\end{proof}

\begin{lemma} \label{l:invisible}
  If $S \subset X$ is such that for all $i \in \N$ and $\gamma \in \tilde{T}_{\delta,i}$, $\cH^1(S \cap \gamma) = 0$, then $\cH^1(S \cap \gamma) = 0$ for all $\gamma \in \tilde{T}_\delta$.
\end{lemma}

\begin{proof}
   Let $\gamma \in \tilde{T}_\delta$. By Lemma \ref{l:banach-partition} with any $\epsilon < \frac{1}{2}$, there is a Borel decomposition $\dom(\gamma) = N \cup E_1 \cup E_2 \cup ...$ so that $\mathcal{L}(N) = 0$ and for almost every $t \in E_i$, we have
   \begin{align*}
       (\xi_i \circ \gamma)'(t) &> \tfrac{1}{2} \Lip(\gamma,t), \\
       |(\phi \circ \gamma)'(t)| &< \delta \Lip(\gamma,t).
   \end{align*}
   It then suffices to prove that $\cH^1(\gamma|_{E_i} \cap S) = 0$ for all $i$.
   
The proof is now essentially that of \cite[Lemma 5.5]{bate-jams}, but we provide details for the reader. Suppose $\cH^1(\gamma|_{E_i}\cap S)>0$.
By standard measure theoretic techniques, there exist $\alpha,\delta', \in \R$ and $0 < a < b \leq 1$ so that $\delta' < \delta, \alpha > \frac{1}{2}$
and there exists a bounded $D \subset E_i$ of positive measure with
\begin{align*}
    \|(\phi \circ \gamma)'(t)\| < \delta'a, \ \ \ (\xi_i \circ \gamma)'(t) > \alpha b, \text{ and } a < \Lip(\gamma,t) < b
\end{align*}
Let $R > 0$ and $D' \subset D$ be a positive measure such that for every $t_0 \in D'$ and $t \in E_i$ with $|t - t_0| \leq R$,
\begin{align*}
    a|t-t_0| \leq \|\gamma(t) - \gamma(t_0)\| &\leq b |t-t_0|
    \\
    |\phi(\gamma(t)) - \phi(\gamma(t_0))| &\leq \delta' a |t -t_0|,
\end{align*}
and for every $s,t \in E_i$ with $t_0 - R < s < t_0 < t < t_0 + R$, we have
\begin{align*}
    \xi_i(\gamma(t_0)) - \xi_i(\gamma(s)) &\geq \alpha b (t_0 - s), \\
    \xi_i(\gamma(t)) - \xi_i(\gamma(t_0)) &\geq \alpha b (t - t_0).
\end{align*}
We let $D_0$ be the intersection of $D'$ with an interval of length $R$ so that $D_0$ has positive measure and $I$ is the smallest interval containing $D_0$. We first extend $\gamma$ to $\overline{D_0}$ and then piecewise {\it linearly} extend to a curve $\tilde{\gamma}:I\to \tilde{X}$. Note that $\tilde{\gamma}$ has the same  Lipschitz constant of $\gamma$.

For any connected component $(c,d)$ of $I \setminus \overline{D_0}$, we can find sequences $c_m,d_m \in D_0$ so that $c_m \to c$, $d_m \to d$. Then
$$\xi_i(\tilde{\gamma}(c)) - \xi_i(\tilde{\gamma}(d)) = \lim_{m \to \infty} \xi_i(\tilde{\gamma}(c_m)) - \xi_i(\tilde{\gamma}(d_m)) \geq \alpha b(d-c).$$
Since our extension of $\gamma$ to $\tilde{\gamma}$ was linear and $\xi_i$ is linear, we get from the Lipschitz bound on $\tilde{\gamma}$ that for any $t < t' \in I$,
$$\xi_i(\tilde{\gamma}(t')) - \xi_i(\tilde{\gamma}(t)) \geq \alpha b (t'-t) \geq \alpha \|\tilde{\gamma}(t) - \tilde{\gamma}(t')\|.$$
Note that this means $\tilde{\gamma}$ is biLipschitz.
As $\phi$ is linear, we can similarly show for any $t, t' \in I$, that
$$|\phi(\tilde{\gamma}(t)) - \phi(\tilde{\gamma}(t'))| \leq \delta' \|\gamma(t) - \gamma(t')\|.$$
Thus $\tilde{\gamma}\in \tilde{T}_{\delta',i}\subset \tilde{T}_{\delta,i}$.
As $0=\cH^1(\tilde{\gamma} \cap S) \geq \cH^1(\tilde{\gamma}(D_0)) > 0$, we get a contradiction.
\end{proof}

\begin{remark} \label{r:linear}
The proof and statement of Lemma 5.5 in \cite{bate-jams} works only if $\phi$ and $\psi$ are linear. This is not a problem there as one may precompose with another embedding so that $\phi$ and $\psi$ are linear, as was done in Definition 5.6 of \cite{bate-jams} (and as we did with $\R^n \oplus_\infty \ell_\infty$).  The authors thank David Bate for pointing this out.
\end{remark}

\begin{lemma} \label{l:invert-LI}
  For any closed linearly independent cones $C_1,...,C_n \subset \R^n$, there is a $M > 0$ so that if $\delta > 0$ and $A$ is an $n\times n$ matrix with columns $a_i \in C_i \cap B(0,\delta)^c$, then $\|A^{-1}\|_{op} \leq M\delta^{-1}$.
\end{lemma}

\begin{proof}
  We first prove the lemma for the special case when $|a_i| = 1$ for all $i$ (so that $\delta = 1$). Then $a_i \in C_i \cap S^{n-1} =: \hat{C}_i$. The function $(a_1,...,a_n) \mapsto \|A^{-1}\|_{op}$ is continuous when defined on $\R^{n \times n} \setminus S$ where $S$ is the closed variety of vectors that give singular matrices. As $\hat{C}_1 \times ... \times \hat{C}_n \subset \R^{n \times n} \setminus S$ is compact, we have that there is a $M > 0$ so that
  $$\max_{(a_1,...,a_n) \in \hat{C}_1 \times ... \times \hat{C}_n} \|A^{-1}\|_{op} \leq M,$$
  as desired.

  Now let $a_i \in B(0,\delta)^c$ with associated matrix $A$ and let $u_i = \frac{a_i}{|a_i|}$ have associated matrix $U$. Let $M > 0$ be the quantity of the lemma for $U$ as established before. Then $A = UD$ where $D$ is the diagonal matrix with entries $|a_i| \geq \delta$. Then $A^{-1} = D^{-1}U^{-1}$, $\|D^{-1}\|_{op} \leq \delta^{-1}$, and so
  \begin{align*}
    \|A^{-1}\|_{op} = \|D^{-1}U^{-1}\|_{op} \leq M\delta^{-1}.
  \end{align*}
\end{proof}

\begin{proof}[Proof of Proposition \ref{p:extend-alberti}]
  Let $C_1,...,C_n \subset \R^n$ be the independent cones for which $\mu$ has Alberti representations that are $\phi$-independent and let $e_j$ be the axis of each $C_j$.
  Applying Corollary 5.9 of \cite{bate-jams} to each $C_j$ (slightly inflated) with $\psi(x) = \langle e_j, \phi(x) \rangle$, we get that there exists a Borel decomposition $X = \bigcup_i A_i$ where each $\mu|_{A_i}$ has Alberti representations in each $\phi$-direction $C_j$ with $\phi$-speed at least some $\delta_i > 0$. It suffices to prove the conclusion of the proposition for each $A_i$. We fix one such $A_i$ and let $A = A_i$, $\delta = \delta_i$.

  Let $\kappa > 0$ be sufficiently small depending on $n$, $\delta$, and $C_1,...,C_n$ to be determined. Let $B \subset A$ be a compact subset so that there is some $k \in \N$ for which
  \begin{align}
    2^{-k-1} \leq \Lip(\phi,x) < 2^{-k}, \qquad \forall x \in B, \label{eq:Lip-bounds}
  \end{align}
  and $\tilde{B}$ be its compact convex hull.  As $\tilde{T}_{\kappa 2^{-k},i}$ are all closed, by repeated application of Lemma 5.2 of \cite{bate-jams}, we get a Borel decomposition
  \begin{align*}
    B = S \cup E_1 \cup E_2 \cup ...
  \end{align*}
  where $\mu|_{E_i}$ supports an Alberti representation on $\tilde{T}_{\kappa 2^{-k},i}$ and $\cH^1(S \cap \gamma) = 0$ for all $\gamma \in \tilde{T}_{\kappa2^{-k},i}$ and $i \in \N$. Lemma \ref{l:invisible} then says that $\cH^1(S \cap \gamma) = 0$ for all $\gamma \in \tilde{T}_{\kappa 2^{-k}}$. As $S \subseteq B$, we get by \eqref{eq:Lip-bounds} that $\cH^1(S \cap \gamma) = 0$ for all $\gamma \in T_\kappa(\phi)$.

  Let $\phi_i = (\phi,\xi_i)$ be $\R^{n+1}$-valued Lipschitz maps. We now show that $\mu|_{E_i}$ supports $n+1$ $\phi_i$-independent Alberti representations. Let $\cA_1,...,\cA_n$ be the Alberti representations of $\mu|_{E_i}$ that are in the $\phi$-direction $C_1,...,C_n$ with $\phi$-speed $\delta$. For each $1 \leq j \leq n$, let $\hat{C}_j = (C_j \setminus B(0,\delta 2^{-k})) \times [-1,1] \subset \R^{n+1}$. It is clear that $\hat{C}_1,...,\hat{C}_n$ are linearly independent and $(\phi_i \circ \gamma)'(t) \in \hat{C}_j$ for almost every $\gamma \in \cA_j$.

  As $\mu|_{E_i}$ supports an Alberti representation $\cA$ on $\tilde{T}_{\kappa 2^{-k},i}$, we get that $(\phi_i \circ \gamma)'(t) \in B(0,\kappa 2^{-k}) \times ([-1,-1/2] \cup [1/2,1]) =: \hat{C}_{n+1}$ for almost every $\gamma \in \cA$. We show that $\hat{C}_1,...,\hat{C}_{n+1}$ are linearly independent when $\kappa$ is sufficiently small. Let $v_1,...,v_{n+1}$ be vectors from each set and suppose for contradiction that $c_1,...,c_{n+1}$ are nonzero scalars for which
  \begin{align*}
    c_1v_1 + ... + c_{n+1} v_{n+1} = 0.
  \end{align*}
  Write $v_j = (a_j,b_j)$ where $a_j \in C_j \setminus B(0,\delta 2^{-k})$ and $b_j \in [-1,1]$ when $1 \leq j \leq n$ and $a_{n+1} \in B(0,\kappa 2^{-k})$ and $|b_{n+1}| \geq \tfrac{1}{2}$. As $a_1,...,a_n$ are linearly independent, it must be that $c_{n+1} \neq 0$. Thus, we may actually assume
  \begin{align*}
    c_1v_1 + ... c_nv_n  = v_{n+1}.
  \end{align*}
  Letting $A$ be the matrix with columns $a_i$, we see that $(c_1,...,c_n) = A^{-1}a_{n+1}$. We then get that
  \begin{align*}
    \max_{1 \leq j \leq n} |c_j| \leq \|A^{-1}\|_{op} |a_{n+1}| \leq M \delta^{-1} 2^k \cdot \kappa 2^{-k}
  \end{align*}
  where $M$ is the constant from Lemma \ref{l:invert-LI} that depends only on $C_1,...,C_n$. Now taking $\kappa < \frac{\delta}{4Mn}$, we get that $|c_j| \leq \frac{1}{4n}$ for all $j$. But now we have that
  \begin{align*}
    |c_1b_1 + ... + c_nb_n| \leq \frac{|b_1| + ... + |b_n|}{4n} \leq \frac{1}{4}.
  \end{align*}
  As $|b_{n+1}| \geq \frac{1}{2}$, it cannot be that $v_1,...,v_{n+1}$ is ever linearly independent.

  The sets $\hat{C}_1,...,\hat{C}_{n+1}$ are closed and linearly independent, but are not cones. However, we can cover each $\hat{C}_j$ with a family of cones $\cC_j$ so that if $K_j \in \cC_j$, then $K_1,...,K_{n+1}$ are linearly independent cones. We then apply Corollary 5.9 of \cite{bate-jams} again with $\phi = \phi_i$ and $\psi = \xi_i$ and take appropriate intersections to get that each $E_i$ admits a countable Borel decomposition on which each piece has $n+1$ $\phi_i$ linearly independent Alberti representations. This shows that $B$ satisfies the conclusion of the proposition.
  
  Note that as $\mu|_A$ supports an Alberti representation with $\phi$-speed $\delta$, it must be that $\Lip(\phi,x) > 0$ for almost every $x \in A$. Thus, by varying $k \in \N$, there exists a countable cover of $A$ by such $B$. As each $B$ is of the desired form, we get the proposition.
\end{proof}

\begin{corollary} \label{c:Alberti-cor}
    Let $(X,d,\mu)$ be a compact metric measure space.  If every $\tilde{D}(n)$ set is $\mu$-null, then $\mu$ has $n+1$ independent Alberti representations.
\end{corollary}

\begin{proof}
    Suppose there is a positive measure subset $S$ that does not admit $n+1$ independent Alberti representations. After possibly passing to a positive measure subset, there then exists a Lipschitz $\phi : S \to \R^h$ for $h \leq n$ admitting $h$ $\phi$-independent Alberti representations, but for which there are no positive measure subsets $S' \subset S$ with $h+1$ independent Alberti representations. Proposition \ref{p:extend-alberti} then gives that $S \in \tilde{D}(\phi) \subset \tilde{D}(n)$.
\end{proof}

\begin{remark}
    Whether or not the converse to Corollary \ref{c:Alberti-cor} holds is an interesting question that we do not answer at this point.
\end{remark}

We can now use Proposition \ref{p:extend-alberti} to get the following characterization of rectifiability via the $\tilde{D}(n)$ sets. It is easy to see that this is just a restatement of Theorem \ref{th:intro-null-char}.

\begin{theorem} \label{th:null-char}
Let $(X,d)$ be a compact metric space so that $0 < \cH^p(X) < \infty$. Then $X$ is $p$-rectifiable if and only if any $\tilde{D}(n)$ set is $\cH^p$-null where $n$ is the largest integer smaller than $p$.
\end{theorem}

\begin{proof}
First assume $X$ is not rectifiable. If $p \notin \N$, then Theorem 2.17 of \cite{bate2020purely} says that the $\cH^p$ measure on subsets of $X$ have at most $n$ independent Alberti representations. On the other hand, if $p \in \N$, as we are trying to find a positive measure $\tilde{D}(n)$ subset, we can take (after possibly passing to a subset) $X$ to be purely $p$-unrectifiable. Theorem 1.1 of \cite{bate-weigt} says that if any subset of $X$ has $p$ independent Alberti representations, then that subset is $p$-rectifiable. Thus, it must be that subsets of $X$ have at most $n=p-1$ independent Alberti representations. Either way, subsets of $X$ have at most $n$ independent Alberti representations. Corollary \ref{c:Alberti-cor} then gives that there is some $\tilde{D}(n)$ set of positive measure.

Suppose now that $X$ is rectifiable so that $p$ is an integer. Assume for contradiction that $S\in \tilde{D}(n)$ has $\cH^p(S)>0$.  
After possibly passing to a positive measure subset, we may assume there is a $\kappa > 0$ and a Lipschitz $\phi:S\to \bR^m$ where $m \leq n$ so that $\cH^1(S\cap \gamma)=0$ for all $\gamma\in T_\kappa(\phi)$.
However since $\cH^p(S)>0$ and $S$ is $p$-rectifiable, we then have a biLipschitz map $f:B\to S$ with $B\subset \bR^p$ is Borel of positive Lebesgue measure.
As $\phi\circ f:B\to \bR^m$ is Lipschitz, there is a $C^1$ map $g:\bR^p\to \bR^m$ and $B'\subset B$ so that $g|_{B'}=\phi\circ f|_{B'}$ and $B'$ has positive Lebesgue measure (see for instance \cite[Theorem 3.1.5]{Federer}).

First assume that there is a point of density $x \in B'$ for which $Dg(x) \neq 0$.
Without loss of generality, we may assume $x = \zero$. Note that
$$\dim(\ker(Dg(\zero))) \geq p - m > 0.$$
For $\theta > 0$, let
$$K_\theta = \{v \in S^{p-1} : d(v,\ker(Dg(\zero))) \leq \theta\}$$
be the $\theta$-neighborhood of $S^{p-1} \cap \ker(Dg(\zero))$. As $\zero$ is a density point of $B'$,  for every $r,\theta > 0$, there exists a $v_{r,\theta} \in K_\theta$ so that, setting $L_{r,\theta}(t) = tv_{r,\theta}$, we have $\cH^1(B' \cap L_{r,\theta}|_{[0,r]}) > 0$. By taking $r,\theta$ sufficiently small depending on the $\|Dg(\zero)\|_{op} > 0$ and the continuity of $Dg(\cdot)$ at $\zero$, we get by an elementary argument that
\begin{align*}|(g \circ L_{r,\theta})'(t)| = |Dg(L_{r,\theta}(t))v_{r,\theta}|&\leq \frac{\kappa}{\|f\|_{lip}\|f^{-1}\|_{lip}}\|Dg(L_{r,\theta}(t))\|_{op} \\
&= \frac{\kappa}{\|f\|_{lip}\|f^{-1}\|_{lip}} \Lip(g,L_{r,\theta}(t)) \Lip(L_{r,\theta},t), \qquad \forall t \in  [0,r].
\end{align*}
As $g|_{B'} = \phi \circ f|_{B'}$, we then have for $\gamma = f \circ L_{r,\theta}$ that
\begin{align*}|(\phi \circ \gamma)'(t)| &= |(g \circ L_{r,\theta})'(t)| \\
&\leq \frac{\kappa}{\|f\|_{lip}\|f^{-1}\|_{lip}} \Lip(g,L_{r,\theta}(t)) \Lip(L_{r,\theta},t) \\
&\leq \frac{\kappa}{\|f\|_{lip}\|f^{-1}\|_{lip}} \Lip(\phi,\gamma(t)) \|f\|_{lip} \Lip(\gamma,t) \|f^{-1}\|_{lip} \\
&= \kappa\Lip(\phi,\gamma(t)) \Lip(\gamma,t), \qquad \forall t \in L_{r,\theta}^{-1}(B') \cap [0,r],
\end{align*}
where the last inequality comes from the chain-rule inequality of Lip.

As $f$ is biLipschitz, so then is $\gamma$. Finally, letting $K$ be a positive measure compact subset of $L_{r,\theta}^{-1}(B') \cap [0,r] \subseteq  \gamma^{-1}(S) \cap [0,r]$, we get that $\gamma|_K \in T_\kappa(\phi)$ and $\cH^1(S \cap \gamma|_K) > 0$, a contradiction.

Now assume every point of density $x \in B'$ has $Dg(x) = 0$. We may assume then that $Dg(x) = 0$ for all $x \in B'$. Assume again that $\zero$ is a density point of $B'$. Then there is some $v \in S^{p-1}$ so that setting $L(t) = tv$ and $\gamma = f \circ L$, we get $\cH^1(B' \cap L|_{[0,1]}) > 0$. As $Dg = 0$ on $B'$, we get
$$|(\phi \circ \gamma)'(t)| = |(g \circ L)'(t)| = |Dg(L(t))v| = 0 \qquad \forall t \in L^{-1}(B').$$
Passing again to a positive measure compact subset $K \subseteq L^{-1}(B') \cap [0,1] \subseteq \gamma^{-1}(S) \cap [0,1]$, we get that $\gamma|_K \in T_\kappa(\phi)$ and $\cH^1(S \cap \gamma|_K) > 0$, another contradiction.
\end{proof}

\section{Lower bounding itineraries}

The goal of this section is Lemma \ref{l:small-lower-bound}, which will be key to ensuring that the cost of an itinerary does not decay significantly if it uses sufficiently small shortcuts. However, this section does not use unrectifiability, cost, or shortcuts as the shortcuts and costs have not been constructed.

\begin{lemma} \label{l:lots-shortcuts}
  Let $0 < \beta < 1$. Let $x,y \in X$, $\cB$ be a collection of disjoint balls with $\sup_{B \in \cB} \diam(B) \leq \rho(x,y)$, and for each $B \in \cB$, choose some $x_B,y_B \in \beta B$. Let $(x_1,...,x_N) \in \cI(x,y)$ and $S = \{i_1 < i_2 < ... < i_m\} \subseteq \{1,...,N-1\}$ be the indices where $\{x_i,x_{i+1}\} = \{x_B,y_B\}$ for some $B \in \cB$. Suppose further that
  \begin{align}
    (x_{i_j},x_{i_j+1}) \text{ and } (x_{i_{j+1}},x_{i_{j+1}+1}) \text{ are associated with different balls of }\mathcal{B}\text{ for all } 1 \leq j \leq m-1. \quad \label{eq:diff-shortcuts}
  \end{align}

  If $\sum_{i \in S} \rho(x_i,x_{i+1}) > \beta \rho(x,y)$, then
  $$\sum_{i \notin S} \rho(x_i,x_{i+1}) \geq \frac{1-\beta}{2\beta} \sum_{i \in S} \rho(x_i,x_{i+1}).$$
\end{lemma}

\begin{proof}
  As $\sup_{B \in \cB} \rho(x_B,y_B) \leq \beta \rho(x,y)$, we get that $\# S \geq \left\lceil \frac{\sum_{i \in S} \rho(x_i,x_{i+1})}{\beta \rho(x,y)} \right\rceil \geq 2$. Let $1 \leq j \leq m-1$ and $B,B' \in \cB$ be the distinct balls associated to $(x_{i_j},x_{i_j+1})$ and $(x_{i_{j+1}},x_{i_{j+1}+1})$, respectively. Note that $j$ exists as $m = \#S \geq 2$. Then
  \begin{multline*}
    \sum_{i = i_j+1}^{i_{j+1}-1} \rho(x_i,x_{i+1}) \geq \rho(x_{i_j+1},x_{i_{j+1}}) \geq \rho(\beta B, \beta B') \geq (1-\beta) \frac{\diam(B) + \diam(B')}{2} \\
    \geq \frac{1-\beta}{2\beta}(\rho(x_{i_j},x_{i_j+1}) + \rho(x_{i_{j+1}},x_{i_{j+1}+1}))
  \end{multline*}
  As $\{i_j+1,...,i_{j+1}-1\} \subseteq S^c$, we then get that
  \begin{align*}
    \sum_{i \notin S} \rho(x_i,x_{i+1}) \geq \frac{1-\beta}{2\beta} \sum_{j=1}^{m-1} (\rho(x_{i_j},x_{i_j+1}) + \rho(x_{i_{j+1}},x_{i_{j+1}+1})) \geq \frac{1-\beta}{2\beta} \sum_{i \in S} \rho(x_{i},x_{i+1}).
  \end{align*}
\end{proof}

\begin{lemma} \label{l:small-lower-bound}
  Let $X$ be a compact metric space, $\phi : X \to \R^h$ a 1-Lipschitz map for which there is $0 < \sigma < 1$ so that
  \begin{align}
    \Lip(\phi,x) \geq \sigma, \qquad \forall x \in X, \label{eq:Lip-sigma-bound}
  \end{align}
  and $0 < \kappa < 1$ so that $\cH^1(\gamma \cap X) = 0$ for all $\gamma \in T_\kappa(\phi)$. Then for any $0 < r < R$, $0 < \beta < \frac{1}{400}$, and $\epsilon > 0$, there exists $0 < \lambda < r$ so that the following holds. Let $x,y \in X$ be so that $d = \rho(x,y)$ satisfies $r < d < R$, let $\cB$ be any collection of disjoint balls for which $\sup_{B \in \cB} \diam(B) \leq \lambda$, and for each $B \in \cB$, let $x_B,y_B \in \beta B$ be so that $y_B \in E_\phi(x,\tfrac{\sigma\kappa}{1000})$.
  Let $(x_1,...,x_N) \in \cI(x,y)$ be an itinerary satisfying condition \eqref{eq:diff-shortcuts} and let $H, S \subseteq \{1,...,N-1\}$ be so that
  \begin{align}
    i \in S &\Longleftrightarrow \{x_i,x_{i+1}\} = \{x_B,y_B\} \quad \text{for some $B \in \cB$}, \notag \\
    i \in H &\Longleftrightarrow 
      x_i \notin E_\phi(x_{i+1},\tfrac{\sigma\kappa}{500}),  \label{eq:H-defn}
  \end{align}

  If 
  $\sum_{i\in H} \rho(x_i,x_{i+1}) < \frac{\sigma \kappa}{1000} 
  \sum_{i\in S} \rho(x_i,x_{i+1})$, then
  $$\sum_{i \notin S} \rho(x_{i},x_{i+1}) \geq (1-\epsilon) d.$$
\end{lemma}

\begin{proof}
  First assume that $\sum_{i \in S} \rho(x_{i},x_{i+1}) \geq d$. As $\beta < \frac{1}{3}$, we get from Lemma \ref{l:lots-shortcuts} that
  $$\sum_{i \notin S} \rho(x_{i},x_{i+1}) \geq \sum_{i \in S} \rho(x_i,x_{i+1}) \geq d,$$
  which gives the lemma.
  Thus, we may assume from now on that $\sum_{i \in S} \rho(x_{i},x_{i+1}) < d$.

  Suppose for contradiction that there are $r, R, \beta, \epsilon$ so that the lemma is false for some $\lambda_n \to 0$.
  We then get (after possibly passing to a subsequence) a sequence of counterexample itineraries for which $\lambda_n \to 0$, $d_n \to d \in [r,R]$, and, for each itinerary of the sequence,
$\sum_{i \in S_n} \rho(x_{i},x_{i+1}) < d_n$ and
$\sum_{i \notin S_n} \rho(x_{i},x_{i+1}) < (1-\epsilon) d_n$.

  We will view $X$ as isometrically embedded in a Banach space $Y$.
  As the total length of each itinerary must be in $[d_n,2d_n]$ and we can also assume $\frac{1}{2}d < d_n < 2d$, we can then parameterize the polygonal curve from each itinerary via $\gamma_n$, a constant speed Lipschitz map from $[0,4d]$ for which we can control the Lipschitz constant i.e. $\frac{1}{8} \leq \Lip(\gamma_n) \leq 1$. Then for any subset $D \subset \dom(\gamma_n)$, we have
  \begin{align}
    \frac{1}{8}|D| \leq \ell(\gamma_n|_D) \leq |D|. \label{eq:length-preserve}
  \end{align}
  
  Let $A_n,H_n'\subset [0,4d]$ be the preimage of the ``intervals" in $S_n,H_n$, respectively. We will view these as both collections of intervals and as the set defined by the union of these intervals. It should be clear based on context which we mean. We do state now that when using $|\cdot|$ (the Lebesgue measure), we will always view $A_n$ as the union of the intervals. 

  After passing to a subsequence, we may take $A$ to be the Hausdorff limit of  the sets $A_n$
  and $\gamma$ the limit of the $\gamma_n$. We then have that $A$ is compact, $|A| \geq |\limsup_n A_n| \geq \epsilon d$, $\Lip(\gamma) \leq 1$, and $\ell(\gamma) \geq d$.

  First note that $\gamma(A) \subset X$. Indeed, we have that $\gamma(\partial A_n) \subset X$ and $A_n \subset N_{\lambda_n}(\partial A_n)$. Additionally, $d_{\rm Haus}(A,A_n) \to 0$. The result now follows from compactness of $X$ and the fact that $\lambda_n \to 0$.

  We claim for any $\delta > 0$, there is an open cover $U$ of $A$ and a $N > 0$ so that $|U\setminus A|<\delta|A|$ and $|U| > 10|A_n|$ for any $n > N$.
  
  Indeed, as $\lim_{r\to 0} |N_r(A) \setminus A| = 0$, we can choose $r$ small enough so that setting $U = N_r(A)$, we get $|U \setminus A| < \delta |A|$. By compactness, we may further assume that $U$ is a finite union of disjoint open intervals $I_1 \cup ... \cup I_m$. We now need to find $N$. As $A_n$ Hausdorff converges to $A$, there exists some $N_0 > 0$ so that $A_n \subseteq U$ for all $n \geq N_0$. 

  First let $A^1,A^2$ be two consecutive intervals of $A_n$ and $J$ the interval between them. Let $B^1,B^2 \in \cB$ be the balls associated to the images of $A^1,A^2$ under $\gamma_n$. By condition \eqref{eq:diff-shortcuts}, these are different. Let $r_1,r_2$ be the radii of $B^1,B^2$. Then as $\beta < \frac{1}{400}$, we get
  \begin{align*}
    |J| \geq (1-\beta) (r_1 + r_2) \geq \frac{1-\beta}{16\beta}(|A^1| + |A^2|) > 20(|A^1| + |A^2|).
  \end{align*}
  It now easily follows that if an open interval $I$ intersects at least two intervals of $A_n$, then
  $$|A_n \cap I| < \frac{|I|}{20}.$$

  We let $1 \leq m' \leq m$ be so that (after a possible reindexing) $\# (I_j \cap A) \geq 2$ if and only if $1 \leq j \leq m'$. Note that $m'$ must exist as otherwise $A$ would be a finite set, contradicting the fact that $|A| \geq \epsilon d$.

  First let $1 \leq j \leq m'$. Then there exists two different points $x,y \in I_j \cap A$. Let $\eta = |x-y| > 0$ and $r' = \min(\eta/100,r)$. As $A_n$ Hausdorff converges to $A$, we get that there is some $N_j > 0$ so that $A \subseteq N_{r'}(A_n)$ and $\lambda_n < r'$ for all $n \geq N_j$. For any interval $I \in A_n$, we have $|I| \leq 8\lambda_n < \frac{\eta}{10}$, and so a simple triangle inequality argument gives that there exist two distinct intervals of $A_n$ that are within $r'$ of $x$ and $y$. As $r' < r$, these intervals also intersect $I_j$. Thus, 
  \begin{align*}
    |A_n \cap I_j| < \frac{|I_j|}{20}, \qquad \forall 1 \leq j \leq m', n \geq N_j.
  \end{align*}

  For $j > m'$, we have that $I_j \cap A$ is either a single point or empty. It then follows easily from the fact that $A_n$ Hausdorff converges to $A$ that $\lim_{n \to \infty} |A_n \cap I_j| = 0$. Choose $N_j > 0$ so that
  \begin{align*}
    |A_n \cap I_j| < \frac{|A_n|}{2m}, \qquad \forall m' < j \leq m, n \geq N_j,
  \end{align*}
  something possible as $|A_n| \geq \epsilon d_n$ always.
  
  Let $n \geq N = \max(N_0,...,N_m)$. As $n \geq N_0$, we get that $U$ covers $A_n$. Combining the last two displayed equations then gives
  \begin{align*}
    |A_n| = \sum_{i=1}^{m'} |A_n \cap I_j| + \sum_{i=m'+1}^m |A_n \cap I_j| < \sum_{i=1}^{m'} \frac{|I_j|}{20} + \sum_{i=m'+1}^m \frac{|A_n|}{2m} \leq \frac{|U|}{20} + \frac{|A_n|}{2}.
  \end{align*}
  This gives $|U| > 10|A_n|$ as desired.

  We now use the claim to construct a sequence $U_j \supseteq A$ so that $|U_j \setminus A| < 2^{-j}|A|$ and (after possibly passing to a subsequence) $|U_j| > 10|A_j|$.
  We then get
  \begin{align}
    |A| \geq 10 \limsup_n |A_n| \geq 10 \limsup_n \sum_{i \in S_n} \rho(x_i,x_{i+1}). \label{eq:A-lower}
  \end{align}

  We now claim that $\ell(\gamma|_A) \geq \frac{1}{100} |A|$. Suppose not to the contrary. Note that
  \begin{align}
    \ell(\gamma) \leq d + \frac{1}{10}|A|. \label{eq:gamma-lower}
  \end{align}
  Indeed, for any $n$, we have
  \begin{align*}
    \ell(\gamma_n) \leq \sum_{i \notin S_n} \rho(x_i,x_{i+1}) + \sum_{i \in S_n} \rho(x_i,x_{i+1}) \leq (1-\epsilon)d_n + \sum_{i \in S_n} \rho(x_i,x_{i+1}) \overset{\eqref{eq:A-lower}}{\leq} d_n + \frac{1}{10}|A|.
  \end{align*}
  We then get \eqref{eq:gamma-lower} by lower semicontinuity of length. 

  As $\ell(\gamma_n|_{U_j}) + \ell(\gamma_n|_{U_j^c}) \leq d_n + \frac{1}{10} |A|$ and
  $$\ell(\gamma_n|_{U_j}) \overset{\eqref{eq:length-preserve}}{\geq} \frac{1}{8} |U_j| \geq \frac{1}{8} |A|$$
  we get for any $n$ that
    $\ell(\gamma_n|_{U_j^c}) \leq d_n - \frac{1}{50}|A|$.
  Taking a limit and using lower semicontinuity, we get
  \begin{align*}
    \ell(\gamma|_{U_j^c}) \leq d - \frac{1}{50}|A|
  \end{align*}
  We also have that
  \begin{align*}
    \ell(\gamma|_{U_j}) \leq |U_j \setminus A| + \ell(\gamma|_A) \leq 2^{-j}|A| + \frac{1}{100} |A|.
  \end{align*}
  The last two inequalities give
  \begin{align*}
    \ell(\gamma) \leq d + \left( 2^{-j} - \frac{1}{100} \right)|A| < d
  \end{align*}
  when $j$ is large enough, which is a contradiction. This gives the claim.

  Note that $\gamma(A) \subset X$ and $\ell(\gamma|_A) \geq \frac{1}{100} |A|$.
  
  We claim that
  \begin{align}
    \ell(\phi \circ \gamma|_A) \geq \frac{\sigma \kappa}{100}|A|. \label{eq:big-phi}
  \end{align}
  Indeed, we have by Lemma \ref{l:kirchheim} that there is a disjoint decomposition $A = N \cup E_1 \cup E_2 \cup ...$ so that $\ell(\gamma|_N) = 0$ and $\gamma|_{E_i}$ are all biLipschitz. If $\ell(\phi \circ \gamma|_A) < \frac{\sigma\kappa}{100}|A| \leq \sigma \kappa \ell(\gamma|_A)$, then there exists some $i$ so that $E_i$ has positive measure and
  \begin{align*}
    \int_{E_i} \Lip(\phi \circ \gamma,t) ~dt = \ell(\phi \circ \gamma|_{E_i}) < \sigma \kappa \ell(\gamma|_{E_i}) = \int_{E_i} \sigma \kappa \Lip(\gamma,t) ~dt.
  \end{align*}
  There then exists a positive measure subset $A \subset E_i$ so that
  \begin{align}
    \Lip(\phi \circ \gamma,t) < \sigma \kappa \Lip(\gamma,t), \qquad \forall t \in A. \label{eq:pointwise-Lip-bound}
  \end{align}
  Letting $K \subset A$ be a positive measure compact subset, we get that $\gamma$ is biLipschitz on $K$ and $\ell(\gamma|_K \cap X) > 0$. As \eqref{eq:Lip-sigma-bound} and \eqref{eq:pointwise-Lip-bound} give that $\gamma|_K \in T_\kappa(\phi)$, we get a contradiction that proves \eqref{eq:big-phi}

  Let $U \supseteq A$ be any open cover so that $|U \setminus A| < \frac{1}{4}|A|$. Then by lower semicontinuity, we know that there exists $N$ so that
  $$\ell(\phi \circ \gamma_n|_U) \geq \frac{1}{2} \ell(\phi \circ \gamma|_U) \overset{\eqref{eq:big-phi}}{\geq} \frac{\sigma \kappa}{200} |A|, \qquad \forall n \geq N.$$
  From \eqref{eq:A-lower}, we have for $n$ sufficiently larger that $|A_n| < \frac{5}{4}|A|$ and so
  \begin{multline*}
    \ell(\phi \circ \gamma_n|_{U \cap (H_n' \cup A_n)}) \leq \sum_{i \in H_n} \rho(x_i,x_{i+1}) + \frac{\sigma \kappa}{1000} \ell(\gamma_n|_{A_n}) \\
    < \frac{\sigma \kappa}{1000}\sum_{i \in S_n} \rho(x_i,x_{i+1}) + \frac{\sigma \kappa}{1000} |A_n| \leq \frac{\sigma \kappa}{500} |A_n| < \frac{\sigma \kappa}{400} |A|.
  \end{multline*}
  This gives that $\ell(\phi \circ \gamma_n|_{U \cap (H_n' \cup A_n)^c}) \geq \frac{\sigma \kappa}{400} |A|$. But as intervals of $(H_n')^c$ satisfy
  $$|\phi(\gamma_n(x)) - \phi(\gamma_n(y))| < \frac{\sigma \kappa}{500} |\gamma_n(x) - \gamma_n(y)| \leq \frac{\sigma \kappa}{500} |x-y|,$$
  we get that
  \begin{align*}
    |U \cap (H_n' \cup A_n)^c| \geq \frac{500}{\sigma \kappa} \ell(\phi \circ \gamma_n|_{U \cap (H_n' \cap A_n)^c}) \geq \frac{5}{4}|A|.
  \end{align*}
  But, as $U$ was a small neighborhood of  $A$
  \begin{align*}
    |U \cap (H_n' \cup A_n)^c| \leq |U| = |A| + |U \setminus A| < \frac{5}{4}|A|,
  \end{align*}
  a contradiction.
\end{proof}

\section{Constructing shortcuts} \label{s:construction}

As our space is unrectifiable, Theorem \ref{th:null-char} gives a positive measure subset $F_0 \subset E$, a Lipschitz $\phi : F_0 \to \R^h$ so that $h < p$, and a $\kappa > 0$ so that $F_0 \in \tilde{C}(\phi,\kappa)$. In fact, the proof of Theorem \ref{th:null-char} even gives that $\phi$ can be chosen so that $\cH^p|_{F_0}$ has $h$ $\phi$-independent Alberti representations. It must then be that $\Lip(\phi,x) > 0$ for almost every $x \in F_0$. Thus, there exists a positive measure subset $F_1 \subset F_0$ and a $\sigma > 0$ so that
\begin{align*}
  \Lip(\phi,x) > \sigma, \qquad \forall x \in F_1.
\end{align*}
We may further suppose without loss of generality that there exists a constant $\ahlfors > 1$ so that
\begin{align*}
  \ahlfors^{-1} < \Theta_*(F_1,x) \leq \Theta^*(F_1,x) < \ahlfors, \qquad \cH^p\text{-a.e.} ~x \in F_1.
\end{align*}
We finally let $F \subset F_1$ be a compact subset of positive measure. Note that $F$ inherits the bounds we have for $F_1$ (at least almost everywhere).

For $R > 0$, we define the Borel set
\begin{align}
  F^R = \left\{ x \in F : \ahlfors^{-1} < \frac{\cH^p(F \cap B(x,r))}{r^n} < \ahlfors, \quad \forall r < R \right\}. \label{eq:Fr-defn}
\end{align}
Then $F^r \subseteq F^s$ when $s \leq r$ and $\lim_{R \to 0} \cH^p(F^R) = \cH^p(F)$.

Let $(\alpha_k)_{k=0}^\infty$, $(\epsilon_k)_{k=0}^\infty$, and $(\theta_k)_{k=0}^\infty$ be sequences of numbers in $(0,1)$ so that
\begin{align*}
  \alpha_n &\to 0, \\
  \prod_{i=1}^\infty (1-\epsilon_k) &\geq \frac{1}{2}, \\
  \theta_k &\leq \min(2^{-k}, \tfrac{\sigma \kappa}{1000}).
\end{align*}
Let $0 < \pad < 10^{-10}\kappa^{2}$. We will inductively define a finite set of shortcuts $\cS_k = \{\{x_i,y_i\}\}_{i=1}^{m_k}$ and scales $\lambda_k > 0$ so that if we define for any $k$
\begin{align*}
  \cS_{\leq k} &:= \bigcup_{j \leq k} \cS_j, \\
  S_{\leq k} &:= \bigcup_{\{x,y\} \in \cS_{\leq k}} \{x,y\}.
\end{align*}
then the following properties (along with others) will be maintained:

\begin{enumerate}[(i)]
  \item\label{item:F-contains} $S_{\leq k} \subseteq F^{\lambda_k}$,
  \item\label{item:S-bound} $\frac{2}{\pad} \max_{\{x,y\} \in \cS_k} \rho(x,y) \leq \lambda_k < \min_{\{x,y\} \in \cS_{k-1}} \rho(x,y)$,
  \item\label{item-shortcut-bounds} the balls $\cB_k = \{B(x_i,\pad^{-1} \rho(x_i,y_i))\}_{i=1}^{m_k}$ are disjoint,
  \item\label{item:shortcut-sep}$\left(\bigcup_{B \in \cB_k} B\right) \cap S_{\leq k-1} = \emptyset$,
  \item\label{item-cover} $\cH^p \left( F^{\lambda_k} \setminus \bigcup_{B \in \cB_k} B \right) < 2^{-k}$.
\end{enumerate}

\begin{remark} \label{r:unique-S}
  Note that \eqref{item:S-bound} gives that
  \begin{align*}
    \lambda_{k+1} < \min_{\{x,y\} \in \cS_k} \rho(x,y) \leq \max_{\{x,y\} \in \cS_k} \rho(x,y) \leq \frac{\pad}{2} \lambda_k.
  \end{align*}
  This tells us that $\lambda_k \to 0$ and the distances of pairs of $\cS_k$ are all well separated across $k$.
  Items \eqref{item:S-bound} and \eqref{item-shortcut-bounds} give that
  \begin{align}
    \max_{B \in \cB_k} \diam(B) \leq \lambda_k. \label{eq:max-diam}
  \end{align}
  We also have from \eqref{item-shortcut-bounds} and \eqref{item:shortcut-sep} that
  \begin{align}
    \rho(x,y) \leq \frac{\pad}{1-\pad} \rho\left(\{x,y\}, S_{\leq k} \setminus \{x,y\} \right), \qquad \forall \{x,y\} \in \cS_k. \label{eq:shortcut-neighbor}
  \end{align}
\end{remark}

Let $\cB_{-1} = \cS_{-1} = \emptyset$ and $\lambda_{-1} = \infty$.
Choose some $\lambda_0 > 0$ so that $F^{\lambda_0}$ has positive measure.
As $h < p$, by Corollary \ref{c:transverse-point} with $X = F^{\lambda_0}$ and $\theta = \theta_0$, there exists a full measure set $G_0 \subseteq F^{\lambda_0}$ so that for every  $x \in G_0$, there exists a sequence $y_{x,k} \in F^{\lambda_0} \setminus \{x\}$ so that
\begin{align*}
  y_{x,k} &\in E_\phi(x,\theta_0), \\
  \lim_{k \to \infty} y_{x,k} &= x, \\
  r_{x,k} := \pad^{-1} \rho(x,y_{x,k}) &< \frac{1}{2}\lambda_0.
\end{align*}
As $\bigcup_{x \in G_0} \bigcup_k B(x,r_{x,k})$ covers $G_0$ and is a Vitali covering, we get from Vitali's covering theorem and the fact that $G_0$ is a full measure subset of $F^{\lambda_0}$ that there is a finite collection $\cS_0 = \{\{x_i,y_{x_i,k_i}\}\}_{i=1}^{m_0}$ of pairs of $F^{\lambda_0}$ and corresponding disjoint balls $\cB_0 = \{B(x_i, \pad^{-1}\rho(x_i,y_{x_i,k_i}))\}_{i=1}^{m_0}$ so that
\begin{align*}
  \cH^p \left( F^{\lambda_0} \setminus \bigcup_{B \in \cB_0} B \right) < \frac{1}{2^0}.
\end{align*}

Now assume inductively that we have defined $\lambda_k > 0$, a finite collection of unordered pairs $\cS_k = \{\{x_i,y_i\}\}_{i=1}^{m_k}$ for which \eqref{item:F-contains}-\eqref{item-cover} are satisfied along with
\begin{enumerate}[(i)]
  \setcounter{enumi}{5}
  \item\label{item:r-decay-1} $\lambda_k$ is the result of Lemma \ref{l:small-lower-bound} with $R = \diam(F)$, $r = \min\{\rho(x,y) : x,y \in S_{k-1}, x\neq y\}$ and $\epsilon = \epsilon_k$,
  \item\label{item:transverse} $y \in E_\phi(x,\theta_k)$ for all $\{x,y\} \in \cS_k$,
\end{enumerate}

We will now define $\lambda_{k+1}$ and $\cS_{k+1}$.
Let $R = \diam(F)$, $r = \min\{\rho(x,y) : x,y \in S_k, x\neq y\}$, and apply Lemma \ref{l:small-lower-bound} with $\epsilon = \epsilon_{k+1}$ to get $\lambda_{k+1} < r$.
By Corollary \ref{c:transverse-point} with $\theta = \theta_{k+1}$ and $X = F^{\lambda_{k+1}} \setminus S_{\leq k}$, there is a full measure set $G_{k+1} \subset F^{\lambda_{k+1}} \setminus S_{\leq k}$ so that for every $x \in G_{k+1}$, there is a sequence $y_{x,j} \in F^{\lambda_{k+1}} \setminus \{x\}$ so that
\begin{align*}
  \lim_{j \to \infty} y_{x,j} &= x, \\
  y_{x,j} &\in E_\phi(x,\theta_k), \\
  \pad^{-1} \rho(x,y_{x,j}) &< \min\left(\rho(x,S_{\leq k}), \tfrac{1}{2}\lambda_{k+1}\right).
\end{align*}
Using Vitali's covering theorem as above, we get a finite collection $\cS_{k+1} = \{\{x_i,y_i\}\}_{i=1}^{m_{k+1}}$ of pairs of $F^{\lambda_{k+1}}$ so that the balls $\cB_{k+1} = \{B(x_i,\pad^{-1}\rho(x_i,y_i))\}_{i=1}^{m_{k+1}}$ are disjoint and satisfy
$$\cH^p\left(F^{\lambda_{k+1}} \backslash \bigcup_{B \in \cB_{k+1}} B \right) < \frac{1}{2^{k+1}}.$$
We can easily see that we get items \eqref{item:F-contains}-\eqref{item:transverse}.

Now that we have constructed $\cS_k$ for all $k \geq 0$, we define the cost function where $c(x,y) = \alpha_k \rho(x,y)$ if $\{x,y\} \in \cS_k$ for some $k$ and $c(x,y) = \rho(x,y)$ otherwise.
Note that the choice of $k$ is unique by Remark \ref{r:unique-S}. We now also define $\cS_{\geq k} = \bigcup_{j \geq k} \cS_j$ so that $\cS = \cS_{\geq 0}$.

We also define a function
\begin{align*}
  L : X \times X &\to \Z_+ \\
  (p,q) &\mapsto \begin{cases}
    k, & \mathrm{if} ~\{p,q\} \in \cS_k, \\
    \infty, & \mathrm{if} ~\{p,q\} \notin \cS.
  \end{cases}
\end{align*}
Again, this is well defined. Given an itinerary $\x = (x_1,...,x_N) \in \cI$, we define
\begin{align*}
  L(\x) = \min_{1 \leq i < N} L(x_i,x_{i+1}).
\end{align*}
By abuse of notation, we also define a function
\begin{align*}
  L : X &\to \Z_+ \\
  x &\mapsto \begin{cases}
    k, &\mathrm{if} ~p \in S_k \\
    \infty, &\mathrm{otherwise}.
  \end{cases}
\end{align*}
It should be clear by context which version of $L$ we are using.

We now define $d$ as in Section \ref{s:preliminaries} with the cost function $c$ defined above.

Let $\x = (x_1,...,x_N) \in \cI_A(x,y)$. For each $n \geq 0$, we let $\{k^{(n)}_i\}_{i=1}^{m_n} \subset 2\N$ denote the indices for which $L(x_i,x_{i+1}) \leq n$. We then define the itinerary
\begin{align*}
  D_n(\x) = (x_1,x_{k^{(n)}_1},x_{k^{(n)}_1+1},x_{k^{(n)}_2},x_{k^{(n)}_2+1},...,x_{k^{(n)}_{m_n}},x_{k^{(n)}_{m_n}+1},x_N) \in \cI_A(x,y),
\end{align*}
which can be viewed as simply the subitinerary of $\x$ that only uses the shortcuts of level at most $n$. By finiteness of $\x$, we have that there is some $M \geq 1$ so that $D_M(\x) = \x$.

\section{David-Semmes regularity}

We will continue to use the notation from Section \ref{s:construction}.

Proposition \ref{l:DS-regularity} below is the main result of this section, which shows that the identity mapping between $(F,\rho)$ and $(F,d)$ is David-Semmes-regular in the sense of Definition 1.3{ of \cite{DS00-regular-mappings}.

\begin{proposition}\label{l:DS-regularity}
  There is a constant $C > 0$ so that for all $x$ and $r > 0$, there exist $y_1,y_2,y_3$ so that
  \begin{align*}
    B_\rho(x,r) \subseteq B_d(x,r) \subseteq B_\rho(y_{1},Cr) \cup B_\rho(y_{2},Cr) \cup B_\rho(y_{3},Cr).
  \end{align*}
\end{proposition}

Note that we have not yet proven that $d$ is actually a metric. This will be proven in Lemma \ref{l:is-metric}. The proof of that uses results of this section, but not the proposition.

We will need the following lemma.

\begin{lemma} \label{l:base}
  There exists a constant $\Ctwo > 0$ so that the following holds. Let $\x \in \cI_A(x,y)$ be so that there is some $k \geq 0$ for which $x,y \in S_{\leq k}$ and $L(\x) \geq k+1$. 
  Then $\rho(x,y) \leq \Ctwo c(\x)$.
\end{lemma}

We will prove Lemma \ref{l:base} in Section \ref{s:lemma52}. The proof is independent of the rest of this subsection.

\begin{corollary}\label{c:big-cost}
    There is a constant $\costconst > 0$ so that the following holds.
  Suppose $\x \in \cI_A(x,y)$ so that $x,y \in S_{\leq k}$, $\{x,y\} \notin \cS_{\leq k}$, and $L(\x) \geq k$. Then $\rho(x,y) \leq \costconst c(\x)$.
\end{corollary}

\begin{proof}
    Let $\x = (x_1,...,x_N)$.
    If $L(\x) \geq k+1$, then the corollary follows from Lemma \ref{l:base}. Thus, assume $L(\x) = k$.
    Let $1 = i_1 < ... < i_M = N$ be so that
    \begin{align*}
        D_k(\x) = (x_{i_1},...,x_{i_M}).
    \end{align*}
    Note that $M \geq 4$ as $L(\x) = k$. Let $\sepconst = \frac{\beta}{1-\beta}$.

    If $M = 4$, then as $\{x,y\} \notin \cS_{\leq k}$, it cannot be that both $x_1 = x_2$ and $x_3 = x_4$. Suppose that $x_1 \neq x_2$. As $x_2 \in S_k$ and $x_1 \in S_{\leq k}$, we have by \eqref{eq:shortcut-neighbor} that $\rho(x_2,x_3) \leq \sepconst \rho(x_1,x_2)$. Thus,
    \begin{align*}
      \rho(x,y) \leq (\sepconst + 1) \rho(x_1,x_2) + \rho(x_3,x_4) \leq (\sepconst + 1) c(\x),
    \end{align*}
    as desired. The case when $x_3 \neq x_4$ is handled similarly.

    Now suppose $M > 4$.
    Then $x_{i_j} \in S_{\leq k}$ for all $j$ and $L(\x[i_j,i_{j+1}]) \geq k+1$ for all $j \in 2\N - 1$. Thus, by Lemma \ref{l:base}, we get
    \begin{multline*}
      \rho(x_{i_{j-1}},x_{i_j}) + \rho(x_{i_j},x_{i_{j+1}}) + \rho(x_{i_{j+1}},x_{i_{j+2}}) \overset{\eqref{eq:shortcut-neighbor}}{\leq} (2\sepconst +1) \rho(x_{i_j},x_{i_{j+1}}) \leq (2 \sepconst + 1) \Ctwo c(\x[i_j,i_{j+1}]), \\
      \forall j \in (2\N-1) \cap \{3,...,M-2\}.
    \end{multline*}
    Note also that $\rho(x_{i_j},x_{i_{j+1}}) \leq \Ctwo c(x_{i_j},x_{i_{j+1}})$ when $j = 1,N-1$.

    Summing over $j$ we get that 
    \begin{align*}
      \rho(x,y) &\leq \rho(x_{i_1},x_{i_2}) + \rho(x_{i_{M-1}},x_{i_M}) + \sum_{j \in (2\N-1) \cap \{3,...,M-2\}} \rho(x_{i_{j-1}},x_{i_j}) + \rho(x_{i_j},x_{i_{j+1}}) + \rho(x_{i_{j+1}},x_{i_{j+2}}) \\
      &\leq \Ctwo c(\x[i_1,i_2]) + \Ctwo c(\x[i_{M-1},i_M]) + (2\sepconst + 1)\Ctwo \sum_{j \in (2\N-1) \cap \{3,...,M-2\}} c(\x[i_j,i_{j+1}]) \\
      &\leq (2 \sepconst + 1) \Ctwo c(\x).
    \end{align*} 
    Setting $\costconst=(2\sepconst + 1)\Ctwo$ we get the corollary. 
\end{proof}

\begin{lemma} \label{l:bounded-geometry}
  There is some $\bddgeom > 0$ so that for all $x$ and $r > 0$, there is at most one pair $\{y_1,y_2\} \in \cS$ so that $\rho(y_1,y_2) \geq \bddgeom r$ and $\{y_1,y_2\} \cap B_d(x,r) \neq \emptyset$.
\end{lemma}

\begin{proof}
  Suppose for contradiction there are two different pairs $\{p_1,p_2\},\{q_1,q_2\} \in \cS$ satisfying the conclusion of the lemma. We may suppose that (without loss of generality) $p_1,q_1 \in B_d(x,r)$. As $p_1,q_1 \in B_d(x,r)$, we get by openness of balls that there is an itinerary $\x = (x_1,...,x_N) \in \cI_A(p_1,q_1)$ so that $c(\x) < 2r$.

  The itinerary $\x$ could use shortcuts of $\rho$-length greater than $\bddgeom r$.
  We take $1 \leq i < j \leq N$ so that $\x' = \x[i,j] \in \cI_A(x',y')$ has $(x',y') \notin \cS$, all shortcuts of $\x'$ have $\rho$-length at most $\bddgeom r$ and $j - i$ is maximal. By maximality of $\x'$ and the assumption that $\{p_1,p_2\}$ and $\{q_1,q_2\}$ have $\rho$-length at least $\bddgeom r$, we may also assume that $x',y'$ are associated to shortcuts of $\rho$-length at least $\bddgeom r$. It then follows that $\max(L(x'),L(y')) \leq L(\x')$.
  Then by Corollary \ref{c:big-cost} we get
  $\rho(x',y')\leq \costconst c(\x') \leq \costconst c(\x)$.

  However, by \eqref{eq:shortcut-neighbor}, we have that
  $\rho(x',y') \geq \frac{1-\beta}{\beta} \bddgeom r$. Thus, by taking $\bddgeom$ sufficiently large, we contradict the fact that $c(\x) < 2r$.  
\end{proof}

\begin{lemma} \label{l:tail-bounds}
  There exists a constant $\tailconst > 0$ so that if $\x = (x_1,...,x_N) \in \cI_A(x,y)$, $L(y) \leq L(\x)$, and $x_{N-1} \neq x_N$, then $\rho(x,y) \leq \tailconst c(\x)$.
\end{lemma}

\begin{proof}
  If $L(\x) = \infty$, then $N = 2$, $\x = (x,y)$ is a non-shortcut, and so $c(\x) = \rho(x,y)$. Thus, assume $L(\x) < \infty$ so that $N \geq 4$.

  Let $k = L(\x)$ and $1 < i^{(1)}_1 < ... < i^{(1)}_{M_1} < N$ be the indices so that $(x_i,x_{i+1}) \in \cS_{k}$. Let $i^{(1)}_{M_1+1} = N$. For each $1 \leq j \leq M_1$, note that $x_{i^{(1)}_j+1},x_{i^{(1)}_{j+1}} \in S_{\leq k}$ and $L(\x[i^{(1)}_j+1,i^{(1)}_{j+1}]) > k$. This holds even for $j = M_1$ by our assumption that $y \in L(y) \leq L(\x) = k$ and $x_{N-1} \neq x_N$. Thus, by Lemma \ref{l:base}, we have that
  \begin{align*}
    \rho(x_{i^{(1)}_j+1},x_{i^{(1)}_{j+1}}) \leq \Ctwo c(\x[i^{(1)}_j+1,i^{(1)}_{j+1}]).
  \end{align*}
  By \eqref{eq:shortcut-neighbor}, we get
  \begin{align*}
    \rho(x_{i^{(1)}_j},x_{i^{(1)}_{j}+1}) \leq \frac{\beta}{1-\beta} \rho(x_{i^{(1)}_{j}+1},x_{i^{(1)}_{j+1}}) \leq \frac{1}{2} \rho(x_{i^{(1)}_j+1},x_{i^{(1)}_{j+1}}).
  \end{align*}
  Thus, we get
  \begin{multline*}
    c(\x[i^{(1)}_j,i^{(1)}_{j+1}]) \geq c(\x[i^{(1)}_j+1,i^{(1)}_{j+1}]) \geq \frac{1}{\Ctwo} \rho(x_{i^{(1)}_j+1},x_{i^{(1)}_{j+1}}) \\
    \geq \frac{1}{2\Ctwo} \left( \rho(x_{i^{(1)}_j},x_{i^{(1)}_j+1}) + \rho(x_{i^{(1)}_{j}+1},x_{i^{(1)}_{j+1}}) \right) \geq \frac{1}{2\Ctwo} \rho(x_{i^{(1)}_j},x_{i^{(1)}_{j+1}}).
  \end{multline*}
  Applying this for all $1 \leq j \leq M_1$ and summing, we get
  \begin{align*}
    \rho(x_{i^{(1)}_1},x_{N}) \leq 2\Ctwo c(\x[i^{(1)}_1,N]).
  \end{align*}

  Note now that $L(x_{i^{(1)}_1}) \leq L(\x[1,i^{(1)}_1])$. If $L(\x[1,i^{(1)}_1]) < \infty$, we can repeat the process above to get some $i^{(2)}_1$ so that
  \begin{align*}
    \rho(x_{i^{(2)}_1},x_{i^{(1)}_1}) \leq 2\Ctwo c(\x[i^{(2)}_1,i^{(1)}_1]).
  \end{align*}
  We keep repeating this process until we get $i^{(M)}_1 < ... < i^{(1)}_1$ so that $L(\x[1,i^{(M)}_1]) = \infty$. This is guaranteed as $\x$ is a finite itinerary. Note then that $c(\x[1,i^{(M)}_1]) \geq \rho(x,x_{i^{(M)}_1})$. Letting $i^{(M+1)}_1 = 1$, $i^{(0)}_1 = N$, and summing, we get
  \begin{align*}
    c(\x) = \sum_{j=0}^M c(\x[i^{(j)}_1, i^{(j+1)}_1]) \geq \frac{1}{2\Ctwo} \sum_{j=0}^M \rho(x_{i^{(j)}_1},x_{i^{(j+1)}_1}) \geq \frac{1}{2\Ctwo} \rho(x,y),
  \end{align*}
  as desired.
\end{proof}

\begin{lemma} \label{l:small-jumps}
  There exists a constant $\Csix > 0$ so that the following holds. Let $\x = (x_1,...,x_N) \in \cI_A(x,y)$ and $r = \max_{i \in 2\N \cap \{1,...,N-1\}} \rho(x_i,x_{i+1})$. Then $\rho(x,y) \leq \Csix c(\x) + 2r$.
\end{lemma}

\begin{proof}
  We will show $\rho(x,y) \leq Cc(\x) + 2r$ for different constants $C$ in a finite number of cases. The lemma then follows by taking the minimum of all the constants $C$. If $N = 2$, the lemma is trivial as $(x,y)$ is not a shortcut so $\rho(x,y) = c(\x)$.

  Suppose $L(y) \leq L(\x)$. If $x_N \neq x_{N-1}$, then by Lemma \ref{l:tail-bounds}, we get $\tailconst \rho(x,y) \leq c(\x)$, as desired.  If $x_N = x_{N-1}$, then $L(x_{N-2}) = L(y)$ and $(x_{N-2},x_{N-1})$ is a shortcut of length at most $r$. Note then that we must have $c(\x[1,N-2]) \geq \tailconst \rho(x,x_{N-2})$ as either $N = 4$ so that $\x[1,N-2] = (x,x_2)$, a non-shortcut, or $x_{N-3} \neq x_{N-2}$ in which case we can use Lemma \ref{l:tail-bounds}. Thus, we get
  \begin{align*}
    \rho(x,y) \leq \rho(x,x_{N-2}) + \rho(x_{N-2},x_{N-1}) \leq \tailconst c(\x[1,N-2]) + r \leq \tailconst c(\x) + r,
  \end{align*}
  as desired.

  Thus, we may now assume $k := L(\x) < L(y)$. Clearly if $L(x) \leq L(\x)$, we can run the same argument from before on the reverse itinerary of $\x$ and so we may also assume $L(x) > L(\x)$.

  Let $1 < i_1 < ... < i_M < N$ be the indices for which $(x_i,x_{i+1}) \in \cS_k$. Note that $x_{i_1-1} \neq x_{i_1}$ and $x_{i_M+1} \neq x_{i_{M}+2}$. As $\min(L(\x[1,i_1]),L(\x[i_M+1,N])) > k$ and $x_{i_j} \in S_k$, we get by Lemma \ref{l:tail-bounds} applied to $\x[1,i_1]$ and the reverse of $\x[i_M+1,N]$ that
  \begin{align*}
    \rho(x,x_{i_1}) &\leq \tailconst c(\x[1,i_1]), \\
    \rho(x_{i_M+1},y) &\leq \tailconst c(\x[i_M+1,N]).
  \end{align*}
  
  If $M = 1$, then
  \begin{align*}
    \rho(x,y) &\leq \rho(x,x_{i_1}) + \rho(x_{i_1},x_{i_1+1}) + \rho(x_{i_1+1},x_N) \\
    &\leq \tailconst c(\x[1,i_1]) + r + \tailconst c(\x[i_{i_1+1},N]) \leq \tailconst c(\x) + r,
  \end{align*}
  as desired.

  If $M > 1$, then $L(\x[i_1+1,i_M]) = k$ and $(x_{i_1+1},x_{i_M}) \notin \cS_{\leq k}$ so Corollary \ref{c:big-cost} gives that $\rho(x_{i_1+1},x_{i_M}) \leq \costconst c(\x[i_1+1,i_M])$. We then get
  \begin{align*}
    \rho(x,y) &\leq \rho(x,x_{i_1}) + \rho(x_{i_1},x_{i_1+1}) + \rho(x_{i_1+1},x_{i_M}) + \rho(x_{i_M},x_{i_M+1}) + \rho(x_{i_M+1},x_N) \\
    &\leq \tailconst c(\x[1,i_1]) + r + \costconst c(\x[i_{i_1+1},i_M]) + r + \tailconst c(\x[i_{i_M+1},N]) \leq (\tailconst + \costconst )c(\x) + 2r,
  \end{align*}
  as desired.
\end{proof}

We now prove Proposition \ref{l:DS-regularity} under the assumption that Lemma \ref{l:base} holds.

\begin{proof}[Proof of Proposition \ref{l:DS-regularity}]
  We will let $y_1 = x$ but not define $y_2,y_3$ yet.
  Let $y \in B_d(x,r)$ and $\x = (x_1,...,x_N) \in \cI_A(x,y)$ so that $c(\x) < r$. We will show that $y \in B_\rho(y_i,Cr)$ for a finite number of cases. This will allow us to get one $C$ at the end. Note that this is equivalent to showing that $\rho(y,y_i) \leq Cr$ for some large $C$. 
  
  First suppose that all shortcuts used by $\x$ have $\rho$-length less than $\bddgeom r$. Then Lemma \ref{l:small-jumps} gives that $\rho(x,y) \leq \Csix c(\x) + 2\bddgeom r \lesssim r$ as desired.

  Now assume there is a $j \in 2\N \cap \{1,...,N\}$ for which $\rho(x_j,x_{j+1}) \geq \bddgeom r$. As $x_j,x_{j+1} \in B_d(x,r)$, we get from Lemma \ref{l:bounded-geometry} that $x_j,x_{j+1}$ must be a unique pair $p_1,p_2 \in B_d(x,r)$. We will prove the proposition with $y_2 = p_1$, $y_3 = p_2$.

  Note that $L(x_j,x_{j+1}) \leq L(\x[j+1,N])$. This gives that $\rho(x_{j+1},y) \leq \tailconst c(\x[j+1,N])$. Indeed, either $j = N-2$ so that $\x[j+1,N] = (x_{j+1},y)$ is a non-shortcut or $j < N-2$ in which case $x_{j+1} \neq x_{j+2}$ and so we get the result from Lemma \ref{l:tail-bounds} applied to the reverse itinerary of $\x[j+1,N]$. We then get
  \begin{align*}
    y \in B_\rho(p_1,\tailconst r) \cup B_\rho(p_2,\tailconst r),
  \end{align*}
  as desired.

  If $p_1,p_2$ do not exist, then setting $y_1=x,y_2=x$ would have given the lemma.
\end{proof}

\subsection{Proof of Lemma \ref{l:base}}\label{s:lemma52}

Let $\Phi(x,y) = |\phi(x) - \phi(y)|$.

\begin{lemma} \label{l:shortcut-level-bound}
  Let $n \geq 1$, $\x = (x_1,...,x_N) \in \cI_A(x,y)$, and $J_n = \{1 \leq i < N : (x_i,x_{i+1}) \in \cS_n\}$. Then
  \begin{align*}
    \sum_{i \in J_n} \rho(x_i,x_{i+1}) \leq \frac{2\beta}{1-\beta} \sum_{u \in S^c(D_n(\x))} \rho(u).
  \end{align*}
\end{lemma}

\begin{proof}
  Let $(y_1,...,y_M) = D_n(\x)$ and $J_n' = \{1 \leq i < M : (y_i,y_{i+1}) \in \cS_n\}$. By construction of $D_n(\x)$, the segments of $\x$ and $\y$ corresponding to $J_n$ and $J_n'$ respectively are the same.

  As $\y$ can use shortcuts of at most level $n$, by \eqref{eq:shortcut-neighbor}, we have for each $i \in J_n'$ that
  \begin{align*}
    \max(\rho(y_{i-1},y_i),\rho(y_{i+1},y_{i+2})) \geq \frac{1-\beta}{\beta} \rho(y_i,y_{i+1}),
  \end{align*}
  (the only cases requiring the max is when $i \in \{2,M-2\}$).

  As each such long neighbor segment lies in $S^c(\y)$ and can be neighbor to at most two segments of $J_n'$, we get
  \begin{align*}
    \frac{1-\beta}{\beta} \sum_{i \in J_n} \rho(x_i,x_{i+1}) \leq 2\sum_{u \in S^c(\y)} \rho(u),
  \end{align*}
  which gives the lemma.
\end{proof}

The following lemma is the key iterative step to proving Lemma \ref{l:base}. It says that either the cost of an itinerary does not decay too much at some future generation (using Lemma \ref{l:small-lower-bound}) or the cost of the ending itinerary will be a sufficiently large fraction of the cost of the current subitinerary.

\begin{lemma} \label{l:one-step}
  Let $0 < \beta < 10^{-10}\kappa^{2}$ and $\x = (x_1,...,x_N) \in \cI_A(x,y)$ be so that there is some $k \geq 0$ for which $x,y \in S_{\leq k}$ and $L(\x) = k+1$. Let $D_{k+1}(\x) = (x_{i_1},...,x_{i_M})$. Then at least one of the following holds:
  \begin{enumerate}[(1)]
    \item for each $\ell \in (2\N - 1) \cap \{1,...,M-1\}$, there is some $j_\ell \geq k+1$ so that
      \begin{align}
        \sum_{\ell \in (2\N-1) \cap \{1,...,M-1\}} \sum_{u \in S^c(D_{j_\ell}(\x[i_\ell,i_{\ell+1}]))} \rho(u) \geq (1-\epsilon_{k+1}) \rho(x,y), \label{eq:big-Sc}
      \end{align}
    \item there is a subset $\cD \subset (2\N - 1) \cap \{1,...,M-1\}$ so that
      \begin{align}
        \sum_{j \in \cD} c(\x[i_j,i_{j+1}]) \geq 10^{-8}\kappa^2 \left( \sum_{j \in \cD} \rho(x_{i_j},x_{i_j+1}) + \sum_{u \in S(D_{k+1}(\x))} \rho(u) \right). \label{eq:big-cost}
      \end{align}
  \end{enumerate}
\end{lemma}

\begin{proof}
  Let $\y = D_{k+1}(\x)$ and $H,S$ be the sets from Lemma \ref{l:small-lower-bound} for $\y$ and $\cB = \cB_{k+1}$ where $x_B,y_B$ are the shortcut points of each ball in $\cB$. Note that $S = S(\y)$ and also $\diam_{B \in \cB} \cB \leq \lambda_{k+1}$ (from \eqref{item:S-bound}). Thus, if $\sum_{i \in H} \rho(y_i,y_{i+1}) < \frac{\kappa}{1000} \sum_{i \in S} \rho(y_i,y_{i+1})$,
  then we get \eqref{eq:big-Sc} by taking $j_\ell = k+1$ along with the the use of Lemma \ref{l:small-lower-bound} with $\epsilon = \epsilon_{k+1}$ in defining $\lambda_{k+1}$ in \eqref{item:r-decay-1}.

  Thus, we may now assume that
  \begin{align}
    \sum_{i \in H} \rho(y_i,y_{i+1}) \geq \frac{\kappa}{1000} \sum_{i \in S} \rho(y_i,y_{i+1}). \label{eq:big-H}
  \end{align}
  Let $1 \leq a_1 < b_1 \leq a_2 < b_2 \leq ... \leq a_M < b_M \leq N$ be indices so that $(x_{a_i},x_{b_i}) = (y_i,y_{i+1})$ for each $i \in H$. Define
  \begin{align*}
    \cD = \left\{i \in H : \sum_{u \in S^c(D_j(\x[a_i,b_i]))} \rho(u) < \frac{4000}{\kappa} \rho(y_i,y_{i+1}), ~\forall j \geq k+1\right\}.
  \end{align*}

  Assume
  \begin{align}
    \sum_{i \in \cD} \rho(y_i,y_{i+1}) \leq \frac{1}{2} \sum_{i \in H} \rho(y_i,y_{i+1}). \label{eq:small-D}
  \end{align}
  For each $\ell \in H \setminus \cD$, let $j_\ell$ be the minimal $j$ so that the condition of $\cD$ is violated.
  Then
  \begin{multline*}
    \sum_{\ell \in H \setminus \cD} \sum_{u \in S^c(D_{j_\ell}(\x[a_\ell,b_\ell]))} \rho(u) \geq \frac{4000}{\kappa} \sum_{i \in H \setminus \cD} \rho(y_\ell,y_{\ell+1}) \overset{\eqref{eq:small-D}}{\geq} \frac{2000}{\kappa} \sum_{i \in H} \rho(y_\ell,y_{\ell+1}) \\
    \overset{\eqref{eq:big-H}}{\geq} \sum_{i \in H} \rho(y_\ell,y_{\ell+1}) + \sum_{i \in S} \rho(y_\ell,y_{\ell+1}).
  \end{multline*}
  For each $j \in V := (2\N - 1) \cap \{1,...,M-1\} \cap H^c$, we let $j_\ell = k+1$. Note that $D_{k+1}(\x[i_\ell,i_{\ell+1}]) = (x_{i_\ell},x_{i_{\ell+1}})$ is a single non-shortcut segment for all odd $\ell$. Thus
  \begin{align*}
    \sum_{\ell \in (2\N-1) \cap \{1,...,M-1\}} &\sum_{u \in S^c(D_{j_\ell}(\x[i_\ell,i_{\ell+1}]))} \rho(u) \\
    &= \sum_{\ell \in H} \sum_{u \in S^c(D_{j_\ell}(\x[i_\ell,i_{\ell+1}]))} \rho(u) + \sum_{\ell \in V} \sum_{u \in S^c(D_{k+1}(\x[i_\ell,i_{\ell+1}]))} \rho(u) \\
    &\geq \sum_{i \in H} \rho(y_i,y_{i+1}) + \sum_{i \in S} \rho(y_i,y_{i+1}) + \sum_{i \in V} \rho(y_i,y_{i+1}) \\
    &\geq \rho(x,y),
  \end{align*}
  which gives \eqref{eq:big-Sc}.

  Thus, assume
  \begin{align}
    \sum_{i \in \cD} \rho(y_i,y_{i+1}) > \frac{1}{2} \sum_{i \in H} \rho(y_i,y_{i+1}). \label{eq:small-A}
  \end{align}
  Let $i \in \cD$ and let
  \begin{align*}
    J_\ell = \{a_i \leq j < b_i : (x_j,x_{j+1}) \in \cS_\ell\}.
  \end{align*}
  As $\sum_{u \in S^c(D_j(\x[a_i,b_i]))} \rho(u) < \frac{4000}{\kappa} \rho(x_{a_i},x_{b_i})$ for all $j$, we get from Lemma \ref{l:shortcut-level-bound} that
  \begin{multline*}
    \sum_{u \in S(\x[a_i,b_i])} \Phi(u) = \sum_{\ell \geq k+1} \sum_{j \in J_\ell} |\phi(x_j) - \phi(x_{j+1})| \overset{\eqref{item:transverse}}{\leq} \sum_{\ell \geq k+1} \theta_\ell \sum_{j \in J_\ell} \rho(x_j,x_{j+1}) \\
    \leq \frac{4000\beta}{(1-\beta)\kappa} \rho(x_{a_i},x_{b_i}) \sum_{\ell \geq k+1} \theta_\ell \leq \frac{4000\beta}{(1-\beta)\kappa} \rho(x_{a_i},x_{b_i}) \overset{\eqref{eq:H-defn}}{\leq} \frac{1}{2} |\phi(x_{a_i}) - \phi(x_{b_i})|.
  \end{multline*}
  by our choice of $\beta$. As $\sum_{j=a_i}^{b_i} |\phi(x_j) - \phi(x_{j+1})| \geq |\phi(x_{a_i}) - \phi(x_{b_i})|$, we get that
  \begin{align*}
    c(\x[a_i,b_i]) \geq \sum_{u \in S^c(\x[a_i,b_i])} |\Phi(u)| > \frac{1}{2} |\phi(x_{a_i}) - \phi(x_{b_i})| \overset{\eqref{eq:H-defn}}{\geq} \frac{\kappa}{1000} \rho(x_{a_i},x_{b_i}).
  \end{align*}
  As this holds for all $i \in \cD$, we have
  \begin{align*}
    \sum_{i \in \cD} c(\x[a_i,b_i]) \geq \frac{\kappa}{1000} \sum_{i \in \cD} \rho(y_i,y_{i+1}) \overset{\eqref{eq:big-H} \wedge \eqref{eq:small-A}}{\geq} \frac{\kappa}{2000} \left( \frac{1}{3} \sum_{i \in \cD} \rho(y_i,y_{i+1}) + \frac{\kappa}{3000} \sum_{i \in S} \rho(y_i,y_{i+1}) \right).
  \end{align*}
  This gives \eqref{eq:big-cost}.
\end{proof}

\begin{proof}[Proof of Lemma \ref{l:base}]
  We may suppose without loss of generality that $k$ is maximal so that $L(\x) = k+1$.
  We will construct a tree whose nodes are pairs $(x_s,x_t)$ where $s < t$. Let $(x,y)$ be the base of the tree. We let its level be 0. Now given a level $j$ node $(x_s,x_t)$, if $t = s + 1$, we do nothing, set $V(x_s,x_t) = \rho(x_s,x_t)$, and $(x_s,x_t)$ becomes a leaf. Otherwise, we inductively assume that there is some $j \geq 0$ so that $x_s,x_t \in S_{\leq j}$ and $L(\x[s,t]) = j+1$, and we let $(x_{i_1},...,x_{i_M}) = D_{j+1}(\x[s,t])$ and $W = (2\N-1) \cap \{1,...,M-1\}$ be the indices of the non-shortcuts.

  As $\x[s,t]$ satisfies the hypotheses of Lemma \ref{l:one-step}, it must either satisfy \eqref{eq:big-Sc} or \eqref{eq:big-cost}. Call $(x_s,x_t)$ a type I node if it satisfies \eqref{eq:big-Sc} and let $P(x_s,x_t) = 0$ and its children be $C_{(x_s,x_t)} = \bigcup_{\ell \in W} S^c(D_{j_\ell}(\x[i_\ell,i_{\ell+1}]))$. Otherwise call it a type II node in which case, if we let $\cD$ be the subset of $W$ from condition \eqref{eq:big-cost}, then we let
  \begin{align*}
    P(x_s,x_t) := \sum_{\ell \in \cD} c(\x[i_\ell,i_{\ell+1}])
  \end{align*}
  and the children of $(x_s,x_t)$ be $C_{(x_s,x_t)} = \{(x_{i_\ell},x_{i_{\ell+1}})\}_{\ell \in W \setminus \cD}$. We let the levels of $C_{(x_s,x_t)}$ be $j+1$. This leads to the lower bound $L(\x[a,b]) \geq j$ for all level $j$ nodes $(x_a,x_b)$. As $\x$ is finite, $\cT$ will also be. Let $\cL$ be the leafs.

  Given a non-leaf node $u$, we let
  \begin{align*}
    V(u) = P(u) + \sum_{v \in C_u} V(v).
  \end{align*}
  For any node $u$, let $D_u$ be all the descendents of $u$. Then it is not hard to see that $V(u) = \sum_{v \in \cL \cap D_u} V(v) + \sum_{v \in D_u} P(v)$. It also easily follows by induction that $V(x_s,x_t) \leq c(\x[s,t])$. The lemma will follow when we show for a level $j$ node $u$ that
  \begin{align}
    V(u) \geq \frac{\kappa^2 \prod_{\ell=j}^\infty (1-\epsilon_\ell)}{10^8} \rho(u) \label{eq:V-compare}
  \end{align}
  as
  \begin{align*}
    c(\x) \geq V(x,y) \geq \frac{\kappa^2 \prod_{\ell=1}^\infty(1- \epsilon_\ell)}{10^8} \rho(x,y) \geq \frac{\kappa^2}{10^9} \rho(x,y).
  \end{align*}

  For leaf nodes $u \in \cL$, \eqref{eq:V-compare} is obvious. Now let $u = (x_s,x_t)$ be a non-leaf node of level $j$. We may inductively suppose \eqref{eq:V-compare} is satisfied for the level $j+1$ nodes in $C_u$. If $u$ is a type I node, then
  \begin{multline*}
    V(u) = \sum_{\ell \in W} \sum_{v \in S^c(D_{j_\ell}(\x[i_\ell,i_{\ell+1}]))} V(v) \overset{\eqref{eq:V-compare}}{\geq} \frac{\kappa^2 \prod_{\ell=j+1}^\infty (1-\epsilon_\ell)}{10^8}\sum_{\ell \in W} \sum_{v \in S^c(D_{j_\ell}(\x[i_\ell,i_{\ell+1}]))} \rho(v) \\
    \overset{\eqref{eq:big-Sc}}{\geq} \frac{\kappa^2 \prod_{\ell=j}^\infty (1-\epsilon_\ell)}{10^8} \rho(u)
  \end{multline*}
  Otherwise, $u$ is a type II node and so
  \begin{align*}
    V(u) = &\sum_{\ell \in \cD} c(\x[i_\ell,i_{\ell+1}]) + \sum_{\ell \in W \setminus \cD} V(x_{i_\ell},x_{i_{\ell+1}}) \\
    \overset{\eqref{eq:big-cost} \wedge \eqref{eq:V-compare}}{\geq} &10^{-8}\kappa^2 \left( \sum_{j \in \cD} \rho(x_{i_j},x_{i_j+1}) + \sum_{u \in S(D_{j+1}(\x[s,t]))} \rho(u) \right) + \frac{\kappa^2 \prod_{\ell=j+1}^\infty (1-\epsilon_\ell)}{10^8} \sum_{\ell \in W \setminus \cD} \rho(x_{i_\ell}, x_{i_{\ell+1}}) \\
    \geq &\frac{\kappa^2 \prod_{\ell=j}^\infty (1-\epsilon_\ell)}{10^8} \left( \sum_{\ell \in W} \rho(x_{i_\ell},x_{i_{\ell+1}}) + \sum_{u \in S(D_{j+1}(\x[s,t]))} \rho(u) \right) \\
    \geq  &\frac{\kappa^2 \prod_{\ell=j}^\infty (1-\epsilon_\ell)}{10^8} \rho(x_s,x_t)
  \end{align*}
  as desired.
\end{proof}

\section{Proof of Theorem \ref{th:main}}

We continue to let $F$ and $d$ be from Section \ref{s:construction}. We will let $f : (F,\rho) \to (F,d)$ be the identity map. 
Theorem \ref{th:main} will follow immediately from the following three lemmas.

\begin{lemma} \label{l:is-metric}
  $(F,d)$ is a metric space. 
\end{lemma}

\begin{proof}
  As symmetry and the triangle inequality are obvious, we only need to verify that $d(x,y) > 0$ for all $x \neq y$. Let $\x \in \cI_A(x,y)$ be so that $c(\x) < d(x,y) + \epsilon$ for some $\epsilon > 0$ to be determined. If the $\rho$-length of all shortcuts used by $\x$ is at most $\frac{1}{4} \rho(x,y)$, then by Lemma \ref{l:small-jumps}, we get $\rho(x,y) \leq \Csix c(\x) + \frac{1}{2} \rho(x,y)$, which gives
  \begin{align*}
    d(x,y) > c(\x) - \epsilon \geq \frac{1}{2\Csix} \rho(x,y) - \epsilon.
  \end{align*}
  If $\x$ uses some shortcut $(x_i,x_{i+1})$ of length at least $\frac{1}{4} \rho(x,y)$ then there is some $n \geq 0$ depending only on $\rho(x,y)$ so that $\{x_i,x_{i+1}\} \in \cS_{\leq n}$. But
  \begin{align*}
    d(x,y) > c(\x) - \epsilon \geq c(x_i,x_{i+1}) - \epsilon \geq \alpha_n \rho(x,y) - \epsilon.
  \end{align*}
  Thus, if we take $\epsilon = \frac{1}{2} \min\left( \frac{1}{2\Csix}, \alpha_n \right) \rho(x,y) > 0$, we get that $d(x,y) > 0$ in both cases, as desired.
\end{proof}

\begin{lemma}
    $\cH^p_d(F) > 0$.
\end{lemma}
\begin{proof}
By the assumption of $F$, it suffices to show that
    $\cH^p_\rho(F) \lesssim     \cH^p_d(F) $.
To prove this, we may also, by abuse of notation, let $\cH^p$ be the spherical Hausdorff $p$-measure of the appropriate metric normalized so that the measure of the unit ball is 1.

Let $\epsilon > 0$ and $\{B_d(x_i,r_i)\}_{i=1}^\infty$ be a cover of $F$ with balls with respect to the metric $d$ so that
\begin{align*}
    \sum_{i=1}^\infty r_i^p \leq \cH^p_d(F) + \epsilon.
\end{align*}
Using Proposition \ref{l:DS-regularity}, for each $i$, there exist points $y_{i,1},y_{i,2},y_{i,3}$ so that
$$B_d(x_i,r_i) \subseteq B_\rho(y_{i,1},Cr_i) \cup B_\rho(y_{i,2},Cr_i) \cup B_\rho(y_{i,3},Cr_i).$$
This gives a natural cover of $F$ using balls with respect to the metric $\rho$.  Thus
$$\cH^p_\rho(F) \leq 3 C^p (\cH^p_d(F)+\epsilon).$$
Taking $\epsilon$ to $0$ gives our lemma.
\end{proof}

\begin{lemma} \label{l:not-bilip}
    For any positive measure $A \subseteq F$, $f|_A$ is not biLipschitz.
\end{lemma}
\begin{proof}
  First note that \eqref{item-cover} of the inductive hypothesis of Section \ref{s:construction} along with the fact that $\cH^p(F^{\lambda_k}) \to \cH^p(F)$ gives that $\cH^p\left(\bigcup_{B \in \cB_k} B \right) \to \cH^p(F)$. This gives that almost every $x\in F$ has infinitely many indices $n$ so that $x\in \bigcup_{B \in \cB_k} B$, {\em i.e.}
$\cH^p\left(F\setminus\limsup_k \bigcup_{B \in \cB_k} B\right)=0$.

  Let $A \subseteq F$ be a positive measure subset and $0 < \epsilon < \frac{\beta}{4}$ be arbitary. Let $x$ be any point of density of $A \cap \limsup_k \bigcup_{B \in \cB_k} B$.
  As $\Theta^*(F,x) < \ahlfors$, for some $r_x>0$ sufficiently small we have for all $r < r_x$ that
\begin{align}
  \cH^p(F\cap B(x,r)) &\leq \ahlfors r^p, \label{eq:low-weight} \\
  \cH^p(A\cap B(x,r)) &\geq (1-\epsilon^p\ahlfors^{-2}/2^p) \cH^p(F \cap B(x,r)). \label{eq:high-density}
\end{align}
All balls are $\rho$-balls.

  Choose some $k$ large enough so that $\alpha_k < \epsilon$, $\lambda_k < \frac{r_x}{2}$, and $x \in \bigcup_{B \in \cB_k} B_k$. Let $B \in \cB_k$ contain $x$ and $r_B$ its radius.
  Then we can find
  $\{y,z\}\in \cS_n$ lying in $\beta B$ so that
  \begin{align}
    \rho(y,z) = \beta r_B \label{eq:rB-eq}
  \end{align}
  and $B(y,\epsilon r_B) \cup B(z,\epsilon r_B) \subseteq B(x,2r_B)$.

  Suppose for contradiction that $A \cap B(y,\epsilon r_B) = \emptyset$. As $y \in F^{\lambda_k}$ and $r_B < \lambda_k < \frac{r_x}{2}$, we get by \eqref{eq:Fr-defn} that $\cH^p(F \cap B(y,\epsilon r_B)) > \ahlfors^{-1} \epsilon^p r_B^p$.
  Thus,
  \begin{align*}
    \frac{\cH^p(A \cap B(x,2r_B))}{\cH^p(F \cap B(x,2r_B))} \leq \frac{\cH^p(F \cap B(x,2r_B)) - \ahlfors^{-1} \epsilon^p r_B^p}{\cH^p(F \cap B(x,2r_B))} \overset{\eqref{eq:low-weight}}{<} 1 - \ahlfors^{-2} \left(\frac{\epsilon}{2}\right)^p,
  \end{align*}
  but this contradicts \eqref{eq:high-density}. Thus, there exists $y' \in B(y,\epsilon r_B) \cap A$. Similarly, there exists $z' \in B(z,\epsilon r_B) \cap A$. By \eqref{eq:rB-eq} and the fact that $\epsilon < \frac{\beta}{4}$, we have
  \begin{align*}
    \rho(y',z') \geq \rho(y,z) - \rho(y,y') - \rho(z,z') \geq \frac{\beta}{2} r_B.
  \end{align*}
  Thus, we get
  \begin{align*}
    d(y',z') \leq d(y,z) + \rho(y,y') + \rho(z,z') \leq \alpha_n \rho(y,z) + 2 \epsilon r_B \overset{\eqref{eq:rB-eq}}{\leq} \left( \epsilon \beta + 2\epsilon \right) r_B \leq \epsilon \frac{2\beta + 4}{\beta} \rho(y',z').
  \end{align*}
  As $\epsilon$ was arbitrary, we get that $f|_A$ is not biLipschitz.
\end{proof}

\bibliographystyle{alpha}
\bibliography{no-bilip-pieces-1dim}
\end{document}